\documentclass[12pt,reqno]{amsart}

\title[]{On the convexity of the entropy along entropic interpolations}
\author{Christian L\'eonard}
\date{October 2013}

\RequirePackage{amsmath,amsthm,enumerate}
\RequirePackage[psamsfonts]{amssymb}

\usepackage[english]{babel}
\usepackage[LGR,T1]{fontenc}

\RequirePackage[colorlinks,linkcolor=blue,citecolor=blue,urlcolor=blue]{hyperref}

\usepackage[usenames,dvipsnames]{pstricks}
\usepackage{epsfig}
\usepackage{pst-grad} 
\usepackage{pst-plot} 

\newtheorem{theorem}{Theorem}
\newtheorem{lemma}[theorem]{Lemma}
\newtheorem{proposition}[theorem]{Proposition}
\newtheorem{corollary}[theorem]{Corollary}
\newtheorem{claim}[theorem]{Claim}

\newtheorem{definition}[theorem]{Definition}
\newtheorem{definitions}[theorem]{Definitions}

\newtheorem{assumptions}[theorem]{Assumptions}

\newtheorem{questions}[theorem]{Open questions}

\theoremstyle{remark}
\newtheorem{remark}[theorem]{Remark}
\newtheorem{remarks}[theorem]{Remarks}

\newtheorem{examples}[theorem]{Examples}

\numberwithin{theorem}{section}

\newcommand{\RR}{\mathbb{R}}
\newcommand{\Rn}{\mathbb{R}^n}

\newcommand{\EE}{\mathbb{E}}
\newcommand{\1}{\textbf{1}}

\newcommand\pf{_{\#}}
\newcommand{\vol}{\mathrm{vol}}

\renewcommand{\ae}{\textrm{-a.e.}}

    \DeclareMathOperator{\dom}{dom}

    \DeclareMathOperator{\Ric}{Ric}

\newcommand\seq[2]{(#1_#2)_{#2\ge1}}

\newcommand\Lim[1]{\lim_{#1\rightarrow\infty}}

\newcommand{\ud}{\frac{1}{2}}

\newcommand\XX{\mathcal{X}}
\newcommand\XXX{\XX^2}
\newcommand\PX{\mathrm{P}(\XX)}
\newcommand\PXX{\mathrm{P}(\XXX)}
\newcommand\PO{\mathrm{P}(\Omega)}
\newcommand\MO{\mathrm{M}_+(\Omega)}

\newcommand\OO{\Omega}

\newcommand\MX{\mathrm{M}_+(\XX)}
\newcommand\MXX{\mathrm{M}_+(\XXX)}

\newcommand\ii{{[0,1]}}
\newcommand\iX{{[0,1]\times\XX}}
\newcommand\IX{\int_{\XX}}
\newcommand\IXX{\int_{\XXX}}
\newcommand\IO{\int_\Omega}
\newcommand\Iii{\int_\ii}
\newcommand{\IRn}{\int _{\Rn}}

\newcommand{\Lf}{\overrightarrow{L}}
\newcommand{\Lb}{\overleftarrow{L}}
\newcommand{\Af}{\overrightarrow{A}}
\newcommand{\Ab}{\overleftarrow{A}}
\newcommand{\Bf}{\overrightarrow{B}}
\newcommand{\Bb}{\overleftarrow{B}}
\newcommand{\Cf}{\overrightarrow{C}}
\newcommand{\Cb}{\overleftarrow{C}}
\newcommand{\Jf}{\overrightarrow{J}}
\newcommand{\Jb}{\overleftarrow{J}}
\newcommand{\Gf}{\overrightarrow{\Gamma}}
\newcommand{\Gb}{\overleftarrow{\Gamma}}
\newcommand{\If}{\overrightarrow{I}}
\newcommand{\Ib}{\overleftarrow{I}}
\newcommand{\Tf}{\overrightarrow{\Theta}}
\newcommand{\Tb}{\overleftarrow{\Theta}}

\newcommand{\rf}{\overrightarrow r}
\newcommand{\rb}{\overleftarrow r}
\newcommand{\bbf}{\overrightarrow b}
\newcommand{\bbb}{\overleftarrow b}

\newcommand{\ts}{\theta^*}
\renewcommand{\aa}{\mathsf{a}}
\newcommand{\bb}{\mathsf{b}}
\newcommand{\cc}{\mathsf{c}}
\newcommand{\Ixy}{\sum _{x\to y} }

\newcommand{\Ixyz}{\sum _{x\to y\to z}}
\newcommand{\Ixyy}{\sum _{x\to y;x\to y'}}

\begin{document}


 \address{Modal-X. Universit\'e Paris Ouest. B\^at.\! G, 200 av. de la R\'epublique. 92001 Nanterre, France}
 \email{christian.leonard@u-paris10.fr}
\thanks{Partly supported by the ANR projects GeMeCoD (ANR 2011 BS01 007 01) and STAB}

\begin{abstract} 
Convexity properties  of the entropy along displacement interpolations are crucial in the Lott-Sturm-Villani theory of  lower bounded curvature of geodesic measure spaces. As discrete spaces fail to be geodesic, an alternate analogous theory is necessary in the discrete setting. 
\\
Replacing   displacement interpolations by entropic ones  allows for developing a rigorous calculus, in contrast with Otto's informal calculus.  When the underlying state space is a Riemannian manifold,  we show  that the first and second derivatives of the entropy as a function of time along entropic interpolations are expressed in terms of the standard   Bakry-\'Emery operators $\Gamma$ and $ \Gamma_2$. On the other hand, in the discrete setting new operators appear. Our approach is probabilistic; it relies on the Markov property and time reversal.
\\
We illustrate these calculations by means of Brownian diffusions on  manifolds and random walks on graphs. We also give a new unified proof, covering both the manifold and graph cases, of a logarithmic Sobolev inequality in connection with convergence to equilibrium.
\end{abstract}

\maketitle 
\tableofcontents


\section*{Introduction}

Displacement convexity of the relative entropy  plays a crucial role in the Lott-Sturm-Villani (LSV) theory of curvature lower bounds of {metric} measure spaces \cite{St06a,St06b, LV09,Vill09} and the related Ambrosio-Gigli-Savaré gradient flow approach \cite{AGS05,AGS12}.  

Let us explain shortly what is meant by displacement convexity of the relative entropy.
Let  $(\XX,d,m)$ be a metric measure space and $\PX$  denote the  space of all Borel probability measures on $\XX$. The distance $d$ is the basic ingredient of the construction of the displacement interpolations and the positive measure $m$ enters the game as being the reference measure of the relative entropy which is defined by  $H( \mu|m):=\IX \log(d \mu/dm)\,d \mu\in (\infty,\infty],$ $ \mu\in\PX$. Its displacement convexity means that along any  displacement interpolation $[ \mu_0, \mu_1] ^{ \textrm{disp}}=( \mu_t) _{ 0\le t\le 1}$ between two probability measures $ \mu_0$ and $ \mu_1$, the  entropy 
\begin{equation*}
H(t):=H( \mu_t|m),\quad 0\le t\le1
\end{equation*}
of the  interpolation as a function of time admits some  convexity lower bound.\\
Let us recall briefly  what displacement interpolations are. 
The Wasserstein pseudo-distance $W_2$ is the square root  of the optimal quadratic transport cost between $ \mu_0$ and $ \mu_1\in\PX$ given by $W_2^2( \mu_0, \mu_1):= \inf _{ \pi} \IXX d^2\, d \pi$ where the  infimum is taken over all the couplings $ \pi\in\PXX$ of $ \mu_0$ and $ \mu_1.$ It becomes a distance  on the subset $\mathrm{P}_2(\XX)$  of all $ \mu\in\PX$ such that $\IX d^2(x_o,x)\, \mu(dx)< \infty,$ for some $x_o\in\XX.$ Displacement interpolations are the geodesics of the metric space $(\mathrm{P}_2(\XX),W_2)$.  They were introduced  in  McCann's PhD thesis \cite{McC94} together with the notion of displacement convexity of the entropy. Related convex inequalities turn out to be functional and geometric inequalities (Brunn-Minkowski, Prekopa-Leindler, Borell-Brascamp-Lieb, logarithmic Sobolev, Talagrand inequalities), see \cite{McC97,CMS01}.
\\
As a typical result, it is known \cite{OV00,CMS01,SvR05} that a Riemannian manifold has a nonnegative Ricci curvature if and only if $t\mapsto H(t|\vol)$ is convex along any displacement interpolation $( \mu_t) _{ 0\le t\le 1}.$

An important hypothesis of the LSV theory is that the metric space is geodesic. This rules discrete spaces out. Something new must be found to develop an analogue of this theory in the discrete setting. Several attempts in this direction were proposed recently. Maas and Mielke  \cite{Maas11,Mie11,EM11,Mie13} have discovered a Riemannian distance on the set of probability measures such that the evolution of continuous-time Markov chains on a  graph are gradient flows of some entropy with respect to this distance. Bonciocat and Sturm  in \cite{BSt09} and Gozlan, Roberto, Samson and Tetali in \cite{GRST12} proposed two different types of interpolations on the set of probability measures on a discrete metric measure graph which play the role of displacement interpolations in the LSV theory. The author recently  proposed in \cite{Leo12c} a general procedure for constructing displacement interpolations on graphs by slowing down   another type of interpolation, called \emph{entropic interpolation}. 

Entropic interpolations are  the main actors of the present  paper. Analogously to the displacement interpolations which solve  dynamical transport problems, they solve dynamical entropy minimization problems. Although  related to the purpose of this article, \cite{Leo12c} is mainly concerned by the  slowing down asymptotic to build displacement, rather than entropic, interpolations. Displacement interpolations are connected to optimal transport, while entropic interpolations are connected to minimal entropy. 

In this article, we consider the {entropic interpolations} in their own right, without slowing down. The goal remains the same: studying convexity of the entropy along these interpolations to derive curvature  lower  bounds of some Markov generators and state spaces. Our main result states that along any entropic interpolation $[ \mu_0, \mu_1]=( \mu_t) _{ 0\le t\le 1}$ between two  probability measures $ \mu_0$ and $ \mu_1$, the first and second  derivatives of the relative entropy $H(t):=H( \mu_t| m)$ of $ \mu_t$ with respect to some reference measure $m$ can be written as 
\begin{equation}\label{eq-52}
H'(t)= \IX (\Tf \psi_t-\Tb \varphi_t)\, d \mu_t,
\qquad
H''(t)= \IX (\Tf_2 \psi_t+\Tb_2 \varphi_t)\, d \mu_t.
\end{equation}
The arrows on $ \Theta$ and $ \Theta_2$ are related to the forward and backward directions of time.
The functions $ \varphi$ and $ \psi$ depend on the endpoints  $ \mu_0$ and $ \mu_1,$ but the nonlinear operators $ \Tf, \Tb, \Tf_2$ and $ \Tb_2$ only depend on some reference $m$-stationary Markov process. For instance,  when this process is the Brownian diffusion with generator $L=( \Delta-\nabla V\cdot\nabla)/2$ on a Riemannian manifold with $m=e ^{ -V}\vol$ its reversing measure, we show that 
\begin{equation}\label{eq-51}
\Tf=\Tb= \Gamma/2,\quad \Tf_2=\Tb_2= \Gamma_2/2
\end{equation}
are respectively half the carré du champ $\Gamma u:= L(u ^2)-2uLu$ and the iterated carré du champ $ \Gamma_2 u:=L \Gamma u-2 \Gamma(u,Lu)$ which were introduced  by Bakry and \'Emery in  their  article \cite{BE85} on hypercontractive diffusion processes. In the discrete case, these operators are not linked to $ \Gamma$ and $  \Gamma_2$ anymore, but they depend  on $L$ in a different manner, see  \eqref{eq-49} and \eqref{eqd-10}. 
In the Riemannian case, $ \Gamma_2$ is related to the Ricci curvature via  Bochner's formula \eqref{eq-50}. This is the main connection between  geometry and  displacement convexity. Because of the tight relation \eqref{eq-51} between $ ( \Theta,\Theta_2)$ and $(\Gamma, \Gamma_2)$ in the Riemannian case, it is natural to expect that $( \Theta, \Theta_2)$ has also some geometric content in the discrete case. Does it allow for an efficient definition of curvature of graphs? This question is left open in the present paper. However, based on $ \Theta$ and $ \Theta_2,$ we give a unified proof of the logarithmic Sobolev inequality on manifolds and a modified version on graphs. This is a clue in favor of the relevance of the entropic interpolations.

An  advantage of the entropic  calculus, compared to its displacement analogue, is that entropic interpolations are regular in general. This is in contrast with displacement interpolations. Otto's informal calculus provides heuristics that necessitate to be rigorously proved by means of alternate methods while on the other hand entropic interpolations allow for a direct rigorous calculus. As a matter of fact, it is proved in \cite{Leo12a,Leo12c} that displacement interpolations are semiclassical limits of entropic interpolations: entropic interpolations are regular approximations of displacement interpolations.

Another pleasant feature of the  entropic  approach is that it is not only available for reversible reference Markov dynamics, but also for stationary dynamics (in the reversible case, the time arrows on $ \Theta$ and $ \Theta_2$ disappear).

The article's point of view is probabilistic. Its basic ingredients are  measures on  path spaces: typically, Brownian diffusions on manifolds and random walks on graphs. The Markov property plays a crucial role and we take advantage of its time symmetry. Recall that it  states that conditionally on the knowledge of the present state, past and future events are independent. In particular, a time-reversed Markov process is still Markov.  Time-reversal stands in  the core of our approach (even when the reference stochastic process is assumed to be reversible); it explains the forward and backward arrows on $ \Theta$ and $ \Theta_2$. In contrast with analytic approaches, very few probabilistic attempts  have been implemented to explore geometric and functional inequalities. In this respect, let us cite  the contributions of Cattiaux \cite{Ca04} and Fontbona and Jourdain \cite{FJ13} where stochastic calculus is central. In the present article, stochastic calculus is secondary. The symmetry of the Markov property allows us to proceed  with basic measure theoretical  technics. 

\subsection*{Outline of the paper}

The article is organized as follows. Entropic interpolations are discussed at   Section \ref{sec-inter}. In particular, we describe their dynamics. This permits us to write their equation of motion and derive our main abstract result \eqref{eq-52} at Section \ref{sec-second}. This abstract result is exemplified  with Brownian diffusion processes at Section \ref{sec-diff} and random walks on a countable graph at Section \ref{sec-RW}. A unified proof of an entropy-entropy production inequality (logarithmic Sobolev inequality) in both the Riemannian and discrete graph settings,  is given at Section  \ref{sec-equilib} in connection with convergence to equilibrium. The stationary non-reversible case is investigated. At Section \ref{sec-questions}, we propose a heuristics based on  thought experiments in order  to  grasp the link between entropic and displacement interpolations. Finally, we address several open questions related to curvature and entropic interpolations, keeping in mind the LSV theory of geodesic measure spaces as a guideline. This article is  a preliminary step and these open questions should be seen as part of a program toward an analogue of the LSV theory that is available in a discrete setting.

\subsection*{Notation}

For any measurable space  $Y$, we denote by $ \mathrm{P}(Y)$ and $ \mathrm{M}_+(Y)$ the sets of all probability measures and positive   measures on $Y$.
Let  $\Omega=D([0,1],\XX)$ be the space of all right-continuous and left-limited paths from the time interval $[0,1]$  taking their values in a Polish state space $\XX$  endowed with its Borel $\sigma$-field. The canonical process  $X=(X_t)_{0\le t\le 1}$ is defined by $X_t(\omega)=\omega_t\in\XX$ for all $0\le t\le 1$ and $\omega=(\omega_t)_{0\le t\le 1}\in\OO$. As usual, $\Omega$ is endowed with the canonical $\sigma$-field $\sigma(X)$. 
For any $\mathcal{T}\subset\ii$ and any positive measure $Q$ on $\OO,$ we denote $X_\mathcal{T}=(X_t)_{t\in \mathcal{T}},$ $Q_\mathcal{T}=(X_\mathcal{T})\pf Q,$ $\sigma(X_\mathcal{T})$ is the $\sigma$-field generated by $X_\mathcal{T}$.
In particular, for each $t\in\ii,$ $Q_t=(X_t)\pf Q\in \MX$ and $Q _{ 01}=(X_0,X_1)\pf Q\in\MXX.$

\section{Entropic interpolation} \label{sec-inter}

Entropic interpolations are defined and their equations of motion are derived.  The continuous  and discrete space cases are both imbedded in the same abstract setting. This prepares next section where the first and second derivatives of $H(t):=H \mu_t|m)$ are calculated.

\subsection*{$(f,g)$-transform of a Markov process}

The $h$-transform of a Markov process  was introduced by Doob \cite{Doob57} in 1957. It is revisited and slightly extended in two directions:
\begin{enumerate}
\item
We consider a two-sided version of the $h$-transform, which we call  $(f,g)$-transform, taking advantage of the invariance of the Markov property with respect to time reversal;
\item
We   extend   the notion of Markov property to unbounded positive measures on $\OO,$ having in mind two typical examples:
\begin{itemize}
\item
The reversible Wiener process on $\Rn,$ i.e.\ the law of the Brownian motion with Lebesgue measure as its  initial law;
\item
The reversible simple random walk on a countable locally finite graph.\end{itemize}
\end{enumerate}

\subsubsection*{Assumptions on $R$}

We specify some $R\in\MO$ which we call the \emph{reference path measure} and we  assume that it is Markov and $m$-stationary\footnote{This stationarity assumption is not necessary and one could replace $m$ by $R_t$ everywhere. It simply makes things more  readable and shortens some computations.} where $m\in\MX$ is a $ \sigma$-finite positive measure on $\XX$. See Definitions \ref{def-markov} and \ref{def-01} at the appendix Section \ref{sec-def}.

\begin{definition}[$(f,g)$-transform] \label{def-fg}
Let $f_0,g_1:\XX\to[0,\infty)$ be two nonnegative   functions  in $L^2(m).$ The $(f,g)$-transform $P\in\PO$ of $R$ associated with the couple of functions $(f_0,g_1)$ is defined by
\begin{equation}	\label{eq-15}
P:=f_0(X_0)g_1(X_1)\,R\in\PO
\end{equation}
with 
$	
\IXX f_0(x)g_1(y)\,R _{01}(dxdy)=1
$	
for $P$ to be a probability measure. 
\end{definition}

Note that $f_0,g_1\in L^2(m)$ implies that $f_0(X_0)g_1(X_1)\in L^1(R).$ Indeed, denoting $F=f_0(X_0)$ and $G=g_1(X_1),$ we see with $R_0=R_1=m$ that $E_RF^2=\|f_0\|^2 _{ L^2(m)},$ $E_RG^2=\|g_1\|^2 _{ L^2(m)} and $ $E_R(FG)=E_R[FE_R(G\mid X_0)]\le \sqrt{E_RF^2 E_R(E_R(G\mid X_0)^2)} \le \sqrt{E_RF^2E_RG^2}=\|f_0\| _{ L^2(m)}\|g_1\| _{ L^2(m)}< \infty$.

\begin{remarks}\label{rem-01}
Let us write some easy facts about $(f,g)$-transforms and $h$-transforms.
\begin{enumerate}[(a)]
\item
Taking $f_0=\rho_0$ and $g_1=1$ gives us $P=\rho_0(X_0)R$ which is the path measure with initial marginal $\mu_0= \rho_0\,m\in\PX$ and the same  forward Markov dynamics as $R.$
\item
Symmetrically, choosing $f_0=1$ and $g_1=\rho_1$ corresponds  to $P=\rho_1(X_1)R$ which is the path measure with final marginal $\mu_1= \rho_1\,m\in\PX$ and the same  backward Markov dynamics as $R.$
\item
Doob's $h$-transform is an extension  of item (b). It is defined by $P=h(X _{ \tau})\,R _{ [0, \tau]}$ where $ \tau$ is a stopping time and $h$ is such that $P$ is  a probability measure.
\end{enumerate}
\end{remarks}

Let us denote the backward and forward Markov transition kernels of $R$ by
\begin{equation*}
\left\{
\begin{array}{lcll}
\rb(s,\cdot;t,z)&:=&R(X_s\in \cdot\mid X_t=z), &0\le s\le t\le 1,\\
\rf(t,z;u,\cdot)&:=&R(X_u\in \cdot\mid X_t=z),&0\le t\le u\le1,
\end{array}
\right.
\end{equation*}
and 
for $m$-almost all $z\in\XX,$
\begin{equation}\label{eq-23}
\left\{\begin{array}{lclcl}
f_t(z) &:=& E_R(f_0(X_0)\mid X_t=z)&=&\IX f_0(x)\, \rb(0,dx;t,z)\\
g_t(z) &:=& E_R(g_1(X_1)\mid X_t=z)&=&\IX\rf(t,z;1,dy)\, g_1(y).
\end{array}\right.
\end{equation}

The relation between $(f_0,g_1)$ and the endpoint marginals $\mu_0, \mu_1\in\PX$ is
\begin{equation}\label{eq-12b}
\left\{
\begin{array}{lcll}
\rho_0&=&f_0g_0, &  m\ae\\ \\
	\rho_1&=& f_1g_1, &  m\ae
\end{array}\right.
\end{equation} 
where the density functions $ \rho_0$ and $ \rho_1$ are defined by
$\left\{ \begin{array}{lcl}
\mu_0&=&\rho_0\, m\\
\mu_1&=& \rho_1\,m
\end{array}\right..
$
The  system of equations \eqref{eq-12b} was exhibited by Schrödinger in 1931 \cite{Sch31} in connection with the entropy minimization problem \eqref{Sdyn} below. 

\subsection*{Entropic interpolation}
Let us introduce the main notion of this article.

\begin{definition}[Entropic interpolation]\label{def-entint}
Let  $P\in\PO$ be the $(f,g)$-transform of $R$ given by \eqref{eq-15}. Its flow of marginal measures
 $$\mu_t:=(X_t)\pf P\in\PX,\quad 0\le t\le 1,$$ is called  the $R$-\emph{entropic interpolation} between $\mu_0$ and $\mu_1$ in $\PX$ and is denoted by
 $$[\mu_0,\mu_1]^R:=(\mu_t)_{t\in\ii}.$$
When the context is clear, we simply write $[ \mu_0,\mu_1],$ dropping the superscript $R.$
\end{definition}

Next theorem tells us that $\mu_t$ is absolutely continuous with respect to the reference measure  $m$  on $\XX.$ We denote by 
 $$
 \rho_t:=d \mu_t/dm,\quad 0\le t\le1,
 $$ 
its density with respect to $m$. 

Identities \eqref{eq-12b}  extend to all $0\le t\le 1$ as  next result shows.

\begin{theorem}\label{resd-01}
Let  $P\in\PO$ be the $(f,g)$-transform of $R$ given by \eqref{eq-15}. Then,  $P$ is Markov and for all  $0\le t\le1,$ $ \mu_t= \rho_t\,m$ with
\begin{equation}\label{eq-24} 
\rho_t = f_t g_t,\quad \  m\ae
\end{equation}  
\end{theorem}

\begin{proof}
Without getting into detail, the proof works as follows. As $m=R_t,$
\begin{align*}
\rho_t(z):= \frac{dP_t}{dR_t}(z)=E_R \left[ \frac{dP}{dR}|X_t=z\right] &=E_R \left[f_0(X_0)g_1(X_1)| X_t=z\right]\\ &=E_R \left[f_0(X_0)| X_t=z\right]E_R \left[g_1(X_1)| X_t=z\right]=:f_tg_t(z).
\end{align*}
The Markov property of $R$ at the last but one equality is essential in this proof. For more detail, 
see \cite[Thm.\,3.4]{Leo12e}.
\end{proof}

With  \eqref{eq-24}, we see that an equivalent analytical definition of $[ \mu_0,\mu_1]^R$ is as follows. 

\begin{theorem}
The entropic interpolation $[ \mu_0,\mu_1]^R$ satisfies $ \mu_t= \rho_t\,m$   with
\begin{equation}\label{eq-02}
\rho_t(z)= \IX f_0(x)\,\rb(0,dx;t,z) \IX  \rf(t,z;1,dy)\,g_1(y),\quad   \textrm{ for } m\ae\  z\in\XX,
\end{equation}
for each $0\le t\le1$.
\end{theorem}

\subsubsection*{Marginal flows of bridges are (possibly degenerate) entropic interpolations} 
Let us have a look at the $R$-entropic interpolation between the Dirac measures $ \delta_x$ and $ \delta_y$. 
\begin{enumerate}[(a)]
\item
When $R _{ 01}(x,y)>0,$ $[ \delta_x, \delta_y]^R$ is the time marginal flow $(R ^{ xy}_t) _{ 0\le t\le1}$ of the bridge $R ^{ xy}.$ Indeed, taking $ \mu_0= \delta_x$ and $ \mu_1= \delta_y,$  a solution $(f_0,g_1)$ of \eqref{eq-12b} is $f_0=\1 _{ x}/R _{ 01}(x,y)$ and $g_1= \1 _{ y}.$ It follows that the corresponding $(f,g)$-transform is $f_0(X_0)g_1(X_1)\,R=R _{ 01}(x,y) ^{ -1}\1 _{ \left\{X_0=x,X_1=y\right\} }\,R=R ^{ xy}.$
\item
When $R _{ 01}(x,y)=0,$ $[ \delta_x, \delta_y]^R$ is undefined. Indeed, Definition \ref{def-entint} implies that $ \mu_0, \mu_1\ll m.$
 Let us take the sequences of functions $f_0^n=c_n ^{ -1} \1 _{ B(x,1/n)}$ and $g_1^n=\1 _{ B(y,1/n)}$ with $B(x,r)$ the open ball centered at $x$ with radius $r$ and $c_n= R _{ 01}(B(x,1/n)\times B(y,1/n))>0$ the normalizing constant   which is assumed to be positive for all $n\ge1$. The corresponding $(f,g)$-transform of $R$ is the conditioned probability measure $P^n=R(\cdot\mid X_0\in B(x,1/n), X_1\in B(y,1/n))$ which converges as $n$ tends to infinity  to the bridge $R ^{ xy}$ under some assumptions, see \cite{ChU11} for instance. As a natural extension, one can see $[ \delta_x, \delta_y]^R$ as the time marginal flow $t\mapsto R ^{ xy}_t$ of the bridge $R ^{ xy}$. 
 \item
 When the transition kernels  are absolutely continuous with respect to $m$, i.e.\ $\rf(t,z;1,dy)=\rf(t,z;1,y)\,m(dy),$ for all $(t,z)\in [0,1)\times\XX$ and $\rb(0,dx;t,z)=\rb(0,x;t,z)\, m(dx),$ for all $(t,z)\in (0,1]\times \XX$,  then the density functions are equal: $\rf=\rb:=r$ (this comes from the $m$-stationarity of $R$),  and \eqref{eq-02} extends as
 \begin{equation*}
\frac{dR ^{ xy}_t}{dm} (z)= \frac{r(0,x;t,z)r(t,z;1,y)}{r(0,x;1,y)},\quad 0<t<1.
 \end{equation*}
 \item
 Note that in general, for any intermediate time $0<t<1,$ $R ^{ xy}_t$ is not a Dirac mass. This is in contrast with the standard displacement interpolation $[ \delta_x, \delta_y] ^{ \mathrm{disp}}$ whose typical form is $ \delta _{ \gamma ^{ xy}_t}$ where $ \gamma ^{ xy}$ is the constant speed geodesic between $x$ and $y$ when this geodesic is unique.
\end{enumerate}

Any entropic interpolation is a mixture of marginal flows of bridges. This is expressed at \eqref{eq-05b}. For a proof, see \cite{Leo12e}.

\subsection*{Schrödinger problem}

The notion of {entropic} interpolation is related to the  entropy minimization  problem \eqref{Sdyn} below.  In order to state this problem properly, let us first recall an informal definition of the relative entropy
\begin{equation*}
H(p|r):=\int \log(dp/dr)\,dp\in (- \infty,\infty]
\end{equation*}
of the probability measure $p$ with respect to the reference $ \sigma$-finite measure $r$; see the appendix Section \ref{sec-def} for further detail.
The dynamical Schrödinger problem associated with the {reference path measure} $R\in\MO$ is 
\begin{equation*}\label{Sdyn}
H(P| R)\to \textrm{min};\qquad P\in\PO: P_0=\mu_0,\ P_1=\mu_1
 \tag{S$_{\mathrm{dyn}}$}
\end{equation*}
It consists of minimizing the relative entropy
$
H(P|R)
$ of the path probability measure $P$ with respect to the reference path measure $R$  subject to the constraint that $P_0= \mu	_0$ and $P_1= \mu_1$
where $\mu_0, \mu_1\in\PX$ are prescribed initial and final marginals.

We shall need a little bit more than $f_0,g_1\in L^2(m)$ in the sequel. The following set of assumptions implies that some relative entropies are finite and will be invoked at Theorem \ref{res-12} below.

\begin{assumptions}\label{ass-02}
In addition to $f_0,g_1\in L^2(m)$ and the normalization condition \\ $\IXX f_0(x)g_1(y)\, R _{ 01}(dxdy)=1,$ the functions $f_0$ and $g_1$ entering the definition of the $(f,g)$-transform $P$ of $R$ given at \eqref{eq-15} satisfy 
\begin{equation}\label{eq-13}
\IXX [\log_+ f_0(x) +\log_+ g_1(y)] f_0(x)g_1(y)\,R _{01}(dxdy)<\infty,
\end{equation}
where $\log_+h:=\1 _{\{ h>1\}}\log h,$ 
 and as a convention $0\log 0=0$.
\end{assumptions}

\begin{theorem}\label{res-12}
Under the Assumptions \ref{ass-02}, the $(f,g)$-transform $P$ of $R$ which is defined at \eqref{eq-15} is the unique solution of  \eqref{Sdyn} where the prescribed constraints $\mu_0= \rho_0\,m$ and  $\mu_1= \rho_1\,m$ are chosen to satisfy \eqref{eq-12b}.
\end{theorem}

\begin{proof}
See \cite[Thm.\,3.3]{Leo12e}.
\end{proof}

Note that  the solution of \eqref{Sdyn} may not be a $(f,g)$-transform of $R$. More detail about the Schrödinger problem can be found in the survey paper \cite{Leo12e}.

\subsubsection*{$[ \mu,\mu]^R$ is not constant in time}
One must be careful when employing the term interpolation as regards {entropic} interpolations, because in general when $ \mu_0= \mu_1=: \mu,$ the interpolation $[ \mu,\mu]^R$ is not constant in time.
 Let us give two examples. 
 \begin{enumerate}[(a)]
 \item
 Suppose that $R _{ 01}(x,x)>0.$ Then the solution of \eqref{Sdyn} with $ \mu_0= \mu_1= \delta_x$ is the bridge $R ^{ xx}$ and $[ \delta_x,\delta_x]^R$ is the marginal flow $(R ^{ xx}_t) _{ 0\le t\le1}$ which is not constant in general.
 \item
Consider  $R$ to be  the reversible Brownian motion on the manifold $\XX$. Starting from $ \mu$  at time $t=0,$ the entropic minimizer $P$ is such that  $ \mu_t:=P_t$ gets closer to the invariant volume measure $m=\vol$ on the time interval [0,1/2] and then goes back to $ \mu$ on the remaining time period [1/2,1]. Let us show this. On the one hand, with Theorem \ref{resd-06} below, it is easy to show that the function $t\in[0,1]\mapsto H(t):=H( \mu_t|\vol)$ is strictly convex whenever $ \mu\not =\vol$. And on the other hand, a time reversal argument tells us that  $H(t)=H(1-t),$ $0\le t\le1.$  The only constant entropic interpolation is $[\vol,\vol]^R.$
 \end{enumerate}

\subsubsection*{$[ \mu,\mu]^{ \textrm{disp}}$ is  constant in time}

Unlike the entropic interpolation $[ \mu,\mu]^R$, McCann's displacement interpolation $[ \mu,\mu] ^{ \textrm{disp}}$ has the pleasant property to be   constant in time. See \eqref{eq-04b} below for the definition of the displacement interpolation and compare with the representation \eqref{eq-05b} of the entropic interpolation.

\subsection*{Forward and backward stochastic derivatives of a Markov measure}

Since $P$ is Markov, its dynamics is characterized by either its forward stochastic derivative and its initial marginal $ \mu_0$ or its backward stochastic derivative and its final marginal $ \mu_1$. Before computing these derivatives, let us  recall some basic facts about these notions.

Let $Q\in\MO$ be a Markov measure. Its forward stochastic derivative $\partial+\Lf^Q$ is defined by 
$$
[\partial_t+\Lf^Q_t](u)(t,z):=\lim _{h\downarrow 0}h ^{-1} E_Q\big(u(t+h,X _{t+h})-u(t,X_t)\mid X_t=z\big)
$$ 
for any measurable function $u:\iX\to\RR$ in the set $\dom\Lf^Q$ for which this limit exists $Q_t\ae$ for all $0\le t< 1$. In fact this definition is only approximate, we give it here as a support for understanding the relations between the forward and backward derivatives. For a precise statement see \cite[\S 2]{Leo11b}. Since the time reversed $Q^*$ of $Q$ is still Markov, $Q$ admits a backward stochastic derivative  $-\partial+\Lb^Q$ which  is defined by 
$$
[-\partial_t+\Lb^Q_t]u(t,z):=\lim _{h\downarrow 0} h ^{-1} E_Q\big(u(t-h,X _{t-h})-u(t,X_t)\mid X_t=z\big)
$$ 
for any measurable function $u:\iX\to\RR$ in the set $\dom\Lb^Q$ for which this limit exists $Q_t\ae$ for all $0< t\le 1$. Remark that 
\begin{equation*}
\Lb^Q_t=\Lf ^{Q^*}_{1-t},\quad 0< t\le 1.
\end{equation*}

 It is proved in  \cite[\S 2]{Leo11b} that these stochastic derivatives are extensions of the extended forward and backward generators  of $Q$ in the semimartingale sense, see  \cite{DM4}. In particular, they offer us a natural way for computing the generators.

\subsection*{Forward and backward dynamics of the entropic interpolation}

The dynamics in both directions of time of $[ \mu_0,\mu_1]^R$ are specified by the stochastic derivatives $\Af:=\Lf^P$ and $\Ab:=\Lb^P$ of  the  Markov measure $P\in\PO$ which appears at Definition \ref{def-entint}.
 For simplicity, we also denote $\Lf^R=\Lf$ and $\Lb^R=\Lb$ without the superscript $R$ and assume that 
 \begin{equation*}
 \Lf_t=\Lf,\quad \Lb_t=\Lb,\quad \forall 0\le t\le1,
 \end{equation*}
 meaning that the forward and backward transition mechanisms of $R$ do not depend on $t$.
 
To derive the expressions of  $\Af$ and $\Ab,$ we need to introduce the carré du champ $ \Gamma$ of both $R$ and its time reversed $R^*$.  The exact definition of  the extended carré du champ $ \Gamma$ and its domain is given at  \cite[Def.\,4.10]{Leo11b}.  In restriction to the   functions $u,v$ on $\XX$ such that $u,v$ and $uv$ are in $\dom \Lf$, we recover the standard definition    
\begin{equation*}
\left\{
\begin{array}{lcl} 
\Gf(u,v)&:=&\Lf (uv)-u\Lf  v-v\Lf  u\\
\Gb (u,v)&:=&\Lb (uv)-u\Lb  v-v\Lb  u
\end{array}\right.
\end{equation*}

When $R$ is a reversible path measure, we denote  $L=\Lf=\Lb$ and $  \Gamma=\Gf=\Gb$, dropping the useless time arrows.

\begin{remark}\label{rem-02}
In the standard Bakry-\'Emery setting \cite{BE85,Ba92}, $\XX$ is a Riemannian manifold  and one considers the self-adjoint Markov diffusion  generator on $L^2(\XX,e ^{ -V}\vol)$   given by
$
\widetilde L= \Delta-\nabla V\cdot\nabla,
$
with $\vol$  the volume measure and $V:\XX\to\RR$ a regular function.
The usual definition of the carré du champ of   $\widetilde L$ is   $\widetilde\Gamma(u,v):=[\widetilde L(uv)-u\widetilde Lv-v\widetilde Lu]/2.$ In the present paper,  $ \Gamma$ is not divided by 2 and we  consider 
$L=\widetilde L/2,
$ 
i.e.\ 
\begin{equation}\label{eq-11}
L=(-\nabla V\cdot \nabla+\Delta)/2,
\end{equation}
which corresponds to an SDE driven by a standard Brownian motion. Consequently,  $ \Gamma(u,v)= \widetilde\Gamma(u,v)=\nabla u\cdot\nabla v.$
\end{remark}

In general, even if $R$ is a reversible path measure, the prescribed marginal constraints enforce a time-inhomogeneous transition mechanism: the forward and backward derivatives $(\partial_t+\Af_t)_{0\le t\le1}$ and $(-\partial_t+\Ab_t)_{0\le t\le1}$ of $P$  depend explicitly on $t$. It is  known  (see \cite{Leo11b} for instance) that for any function $u:\XX\to\RR$ belonging to some class  of regular functions to be made precise,

\begin{equation}\label{eqd-05}
\left\{
\begin{array}{lcll} 
\Af_t u(z)&=&\displaystyle{\Lf  u(z)+\frac{\Gf (g_t,u)(z)}{g_t(z)}},& (t,z)\in [0,1)\times\XX,\\
\Ab_t u(z)&=&\displaystyle{\Lb  u(z)+\frac{\Gb (f_t,u)(z)}{f_t(z)}},& (t,z)\in (0,1]\times\XX,
\end{array}
\right.
\end{equation}
where $f$ and $g$ are defined at \eqref{eq-23}.
Because of \eqref{eq-24}, for any $t$ no division by zero occurs $\mu_t\ae$ 
For \eqref{eqd-05} to be meaningful, it is necessary that the functions $f$ and $g$ are regular enough for $f_t$ and $g_t$ to be in the domains of the carré du champ operators. 
In the remainder of the paper, we shall only be  concerned with Brownian diffusion processes and random walks  for which the class of regular functions will be specified, see Sections \ref{sec-diff} and \ref{sec-RW}. 

Going back to \eqref{eq-23}, we see that the processes $f_t(X_t)$ and $g_t(X_t)$ are respectively backward and forward  local $R$-martingales. In terms of stochastic derivatives, this is equivalent to 
\begin{equation}\label{eq-27}
\left\{
\begin{array}{ll}
(-\partial_t +\Lb ) f(t,z)=0, &0<t\le 1,\\
f_0,	&t=0,
\end{array}\right.
\qquad
\left\{
\begin{array}{ll}
(\partial_t +\Lf ) g(t,z)=0, &0\le t<1,\\
g_1,	&t=1,
\end{array}\right.
\end{equation}
where these identities hold $ \mu_t$-almost everywhere for almost all $t.$

\begin{assumptions}\label{ass-03}
We assume that the kernels $\rf, \rb$ that appear in \eqref{eq-23} are (a) positivity improving and (b) regularizing:
\begin{enumerate}[(a)]
\item
For all $(t,z)\in(0,1)\times\XX$, $\rf(t,z;1,\cdot)\gg m$ and $\rb(0,\cdot;t,z)\gg m.$
\item
The functions $(f_0,g_1)$ and the kernels $\rf, \rb$  are such that 
$(t,z)\mapsto f(t,z), g(t,z)$ are twice $t$-differentiable \emph{classical solutions} of the  parabolic PDEs \eqref{eq-27}.
\end{enumerate}
\end{assumptions}

Assumption (a) implies that for any $0< t< 1,$ $f_t$ and $g_t$ are positive everywhere. In particular,  we see with \eqref{eq-24} that $ \mu_t\sim m$. 
Assumption (b) will be made precise later in specific settings. It will be used when computing time derivatives of $t\mapsto \mu_t.$

Mainly because we are going to study the relative entropy $H(\mu_t|m)=\IX\log(f_tg_t)\,d\mu_t,$ it is sometimes worthwhile to express $\Af$ and $\Ab$ in terms of the logarithms of $f$ and $g:$ 
\begin{equation}\label{eq-30}
\left\{\begin{array}{ll}
\varphi_t(z):=\log f_t(z)=\log E_R(f_0(X_0)\mid X_t=z), &(t,z)\in (0,1]\times\XX\\
\psi_t(z):=\log g_t(z)=\log E_R(g_1(X_1)\mid X_t=z),&(t,z)\in [0,1)\times\XX.
\end{array}\right.
\end{equation}
In analogy with the Kantorovich potentials which appear in the optimal transport theory, we call $\varphi$ and $\psi$ the \emph{Schrödinger potentials}. Under Assumption \ref{ass-03}-(b), they are classical solutions of the ``second order'' Hamilton-Jacobi-Bellman (HJB) equations 
\begin{equation}\label{eqd-02}
\left\{
\begin{array}{ll}
(-\partial_t +\Bb ) \varphi=0, &0<t\le 1,\\
\varphi_0=\log f_0,	&t=0,
\end{array}\right.
\ 
\left\{
\begin{array}{ll}
(\partial_t +\Bf ) \psi=0, &0\le t<1,\\
\psi_1=\log g_1,	&t=1,
\end{array}\right.
\end{equation}
where the non-linear operators $\Bf $ and $\Bb $ are defined by
\begin{equation*}
\left\{\begin{array}{lcl}
\Bf  u&:= &e ^{-u}\Lf  e^u,\\ \Bb  v&:=& e ^{-v}\Lb  e^v,
\end{array}\right.
\end{equation*}
for any functions $u,v$ such that $e^u\in\dom\Lf $  and  $e^v\in\dom\Lb .$
Let us introduce the notation 
\begin{equation*}
\left\{ \begin{array}{lcl}
\Af _{ \theta}&:=&\Lf + e ^{ -\theta}\Gf (e ^ \theta,\cdot)\\
\Ab _{  \theta}&:=&\Lb + e ^{ -\theta}\Gb (e ^ \theta,\cdot)
\end{array}\right.
\end{equation*}
which permits us to rewrite \eqref{eqd-05} as  $\Af_t=\Af _{ \psi_t}$ and $\Ab_t=\Ab _{ \varphi_t},$ emphasizing their dependence on the Schrödinger potentials.

\subsubsection*{Forward-backward systems} The complete dynamics of $[ \mu_0,\mu_1]^R$ is described by a forward-backward system.   We see that
\begin{equation}\label{eq-08}
\left\{ \begin{array}{ll}
(-\partial_t+\Af ^*_{ \psi_t}) \mu =0, & 0<t\le1,\\
\mu_0, & t=0,
\end{array}\right.
\quad\textrm{ where }\quad
\left\{
\begin{array}{ll}
(\partial_t +\Bf ) \psi=0, &0\le t<1,\\
\psi_1=\log g_1,	&t=1,
\end{array}\right.
\end{equation}
and $A ^*$ is the algebraic adjoint of $A:$ $ \IX Au\,d\mu  =: \left\langle u,A^* \mu \right\rangle .$ The evolution equation for $ \mu$ is understood in the weak sense, in duality with respect to a large enough  class of regular functions $u$. Similarly,
\begin{equation}\label{eq-09}
\left\{ \begin{array}{ll}
(\partial_t+\Ab ^*_{  \varphi_t}) \mu =0, & 0\le t<1,\\
\mu_1, & t=1,
\end{array}\right.
\quad\textrm{ where }\quad
\left\{
\begin{array}{ll}
(-\partial_t +\Bb ) \varphi=0, &0< t\le 1,\\
\varphi_0=\log f_0,	&t=0.
\end{array}\right.
\end{equation}
It appears that the boundary data for the systems \eqref{eq-08} and \eqref{eq-09} are $( \mu_0, g_1)$ and $(f_0, \mu_1)$.

\section{Second  derivative of the entropy}\label{sec-second}

The aim of this section is to provide basic formulas to study the convexity of the entropy 
$$
t\in\ii\mapsto H(t):=H(\mu_t|m)\in [0,\infty)
$$ 
as a function of time, along the entropic interpolation $[\mu_0,\mu_1]^R$ associated with the $(f,g)$-transform $P$ of $R$ defined by \eqref{eq-15}. The interpolation $[\mu_0,\mu_1]^R$ is specified by its endpoint data $(\mu_0,\mu_1)$ defined by \eqref{eq-12b}. We have also seen at \eqref{eq-08} and \eqref{eq-09} that it is specified by either $(\mu_0,g_1)$ or $(f_0,\mu_1)$.

\subsection*{Basic rules of calculus and notation}

 We are going to use the subsequent  rules of calculus and notation where we drop the subscript $t$ as often as possible. Assumption \ref{ass-03} is used in a significant way: it allows to work with everywhere defined derivatives. 
\begin{itemize}
\item
$\rho_t=f_tg_t\,m,$\quad $\varphi:=\log f,$\quad $\psi:=\log g.$
\\
The first identity is Theorem \ref{resd-01}. The others are notation which are well defined everywhere under Assumption \ref{ass-03}.

\item
$\dot u:=\partial_t u,\quad \dot \mu:=\partial_t \mu,\quad \langle u,\eta \rangle :=\IX u\,d \eta.$ 
\\
These are simplifying notation.

\item
The first line of the following equalities is \eqref{eqd-05} and  the other ones are definitions:
\\
$\left\{\begin{array}{lcl}
\Af u&=&\Lf u+\Gf(g,u)/g\\ \Bf u&:=& e ^{-u}\Lf e^u \\
\Cf u&:=&\Bf u-\Lf u
\end{array}\right.,$ \qquad
$\left\{\begin{array}{lcl}
\Ab u&=&\Lb u+\Gb(f,u)/f\\ \Bb u&:=&e ^{-u}\Lb e^u\\
\Cb u&:=&\Bb u-\Lb u.
\end{array}\right.$

\item
$\langle u,\dot \mu\rangle=\langle \Af u,\mu \rangle =\langle -\Ab u,\mu \rangle .$
\\
These identities hold since $\mu$ is the time-marginal flow of the Markov law $P$ whose forward and backward derivatives  are $\Af$ and $\Ab.$
\item
A short way for writing \eqref{eq-27} and \eqref{eqd-02} is:
\\
$\left\{\begin{array}{lcl}
\dot g&=&-\Lf g\\ \dot \psi&=&-\Bf \psi
\end{array}\right.,$ 
\qquad
$\left\{\begin{array}{lcl}
\dot f&=&\Lb f\\ \dot \varphi&=&\Bb \varphi.
\end{array}\right.$
\end{itemize}

\subsection*{Entropy production}

By Assumption \ref{ass-03}, the evolution PDEs \eqref{eq-27} and \eqref{eqd-02} are defined in the classical sense and the functions $\If$ and $\Ib$ below are well defined.

\begin{definition}[Forward and backward entropy production]
For each $0\le t\le 1,$ we define respectively the forward and backward entropy production at time $t$ along the interpolation $[\mu_0,\mu_1]^R$ by
\begin{equation*}
\left\{\begin{array}{lcl}
\If(t)&:=&\IX \Tf  \psi_t \,d\mu_t\\ \\
\Ib(t)&:=&\IX \Tb  \varphi_t\, d\mu_t
\end{array}\right.
\end{equation*}
where for any regular enough function $u,$
\begin{equation}\label{eq-49}
\left\{\begin{array}{lcl}
\Tf  u&:=& e ^{-u}\Gf (e ^{u},u)-\Cf  u\\ \\
\Tb  u&:=& e ^{-u}\Gb (e ^{u},u)-\Cb  u.
\end{array}\right.
\end{equation}
\end{definition}
Calling the functions $\If$ and $\Ib$ ``entropy production'' is justified by Corollary  \ref{res-02} below.

 We shall see at Section \ref{sec-diff} that in the reversible Brownian diffusion setting where $R$ is  associated with the generator \eqref{eq-11}, we have $$\Tf u=\Tb u=\Gamma(u)/2$$ with no dependence on $t$ and as usual, ones simply writes $ \Gamma(u):= \Gamma(u,u).$

\begin{proposition}\label{resd-03}
Suppose that the Assumptions \ref{ass-02} and \ref{ass-03} hold.
The first  derivative of $t\mapsto H(t)$ is
$$
\frac{d}{dt} H(\mu_t|m)=\If(t)-\Ib(t),\qquad 0<t<1.
$$
\end{proposition}

\begin{proof}
With Theorem \ref{resd-01} and the definition of the relative entropy, we immediately see that $H(t)=\left\langle \log \rho,\mu \right\rangle= \langle \varphi+\psi,\mu\rangle .$ It follows with our basic  rules of calculus that
\begin{equation*}
H'(t)=\left\langle \dot \varphi+\dot \psi,\mu \right\rangle + \left\langle \varphi+\psi,\dot \mu \right\rangle =\langle \Bb \varphi-\Bf \psi,\mu\rangle+ \left\langle -\Ab \varphi+\Af \psi,\mu \right\rangle.
\end{equation*}
The point here is to apply the forward operators $\Af=\Af _{ \psi}, \Bf$  to $\psi$ and the backward operators $\Ab=\Ab _{ \varphi},\Bb$ to $\varphi.$
This is the desired result since $(\Af-\Bf)\psi=\Tf \psi$ and $(\Ab-\Bb)\varphi=\Tb \varphi.$
\end{proof}

Remark that $\left\langle \dot \varphi+\dot \psi,\mu \right\rangle= \left\langle  \partial_t \log \rho, \mu\right\rangle = \left\langle \dot \rho/ \rho, \mu \right\rangle = \left\langle \dot \rho,m \right\rangle =(d/dt) \left\langle \rho,m \right\rangle =0.$ Hence,
$$\left\langle \dot \varphi+\dot \psi,\mu \right\rangle=\langle \Bb \varphi-\Bf \psi,\mu\rangle=0.$$

As an immediate consequence of this Proposition \ref{resd-03}, we obtain the following Corollary \ref{res-02} about the dynamics of heat flows. 

\begin{definitions}[Heat flows] \label{def-02}\ 
\begin{enumerate}[(a)]
\item
We call \emph{forward  heat flow} the time marginal flow of the $(f,g)$-transform  described at Remark \ref{rem-01}-(a). It corresponds to $f_0= \rho_0$  and $g_1=1,$ i.e.\  to $P= \rho_0(X_0)\, R.$
\item
We call \emph{backward  heat flow} the time marginal flow of the $(f,g)$-transform  described at Remark \ref{rem-01}-(b). It corresponds to $f_0= 1$  and $g_1= \rho_1,$ i.e.\ to $P= \rho_1(X_1)\, R.$
\end{enumerate}
\end{definitions}

\begin{corollary}\label{res-02}\ 
\begin{enumerate}[(a)]
\item
Along any forward heat flow $ (\mu_t)_{ 0\le t\le 1}$, we have: $\frac{d}{dt} H(\mu_t|m)=-\Ib(t),$ $0<t<1.$
\item
Along any backward heat flow $( \mu_t)_{ 0\le t\le 1}$, we have: $\frac{d}{dt} H(\mu_t|m)=\If(t),$ $ 0<t<1.$
\end{enumerate}
\end{corollary}

\subsection*{Second derivative}
The computations in this subsection are informal and the underlying regularity hypotheses  are kept fuzzy. We  assume that the Markov measure $R$ is \emph{nice enough} for the Schrödinger potentials $\varphi$ and $\psi$ to be in the domains of the compositions of the operators $L, \Gamma, A,B $ and $C$ which are going to appear below. To keep formulas into a reasonable size, we have assumed that  $R$ is  $m$-stationary.
 The informal results that are stated below at Claims \ref{resd-02} and \ref{resd-04}  will turn later at Sections \ref{sec-diff} and \ref{sec-RW}  into formal statements in specific settings.
  We introduce 
\begin{equation}\label{eqd-10}
\left\{
\begin{array}{lcl}
\Tf _{2} u&:=&\Lf  \Tf  u
	+e ^{-u}\Gf \left(e^u,\Tf  u\right)
	+ e ^{-u}\Gf (e^u,u)\Bf  u -e ^{-u}\Gf (e^u\Bf  u,u),\\
\Tb _{2} u&:=&\Lb  \Tb  u
	+e ^{-u}\Gb \left(e^u,\Tb  u\right)+ e ^{-u}\Gb (e^u,u)\Bb  u-e ^{-u}\Gb (e^u\Bb  u,u),
\end{array}
\right.
\end{equation}
provided that the function $u$ is such that these expressions are well defined. 
It will be shown at Section \ref{sec-diff}, see \eqref{eq-38}, that in the  Brownian diffusion setting where $R$ is  associated with the generators $ \left\{ \begin{array}{lcl}
\Lf&=&\bbf\cdot\nabla+ \Delta/2\\
\Lb&=&\bbb\cdot\nabla+ \Delta/2
\end{array}\right., $  we have $ \left\{ \begin{array}{lcl}
\Tf_2&=&\Gf_2/2\\
\Tb_2&=&\Gb_2/2
\end{array}\right. ,$ where
\begin{equation}\label{eqd-07}
\left\{
\begin{array}{lcl}
\Gf_2(u)&:=&\Lf \Gf(u)-2\Gf(\Lf u,u),\\
\Gb_2(u)&:=&\Lb\Gb(u)-2\Gb(\Lb u,u),
\end{array}•\right.
\end{equation}
are   the forward and backward iterated carré du champ operators.

\begin{remark}
We keep the notation of Remark \ref{rem-02}.
In the standard Bakry-\'Emery setting, the iterated carré du champ of   $\widetilde L= \Delta-\nabla V\cdot\nabla$ is   $\widetilde\Gamma_2(u)=[\widetilde L(\widetilde\Gamma(u))- 2\widetilde\Gamma (u,\widetilde Lu)]/2.$ In the present paper, we  consider $L=\widetilde L/2$ instead of $\widetilde L$ and we have already checked at Remark \ref{rem-02} that $ \widetilde\Gamma= \Gamma.$ Consequently,  $ \Gamma_2= \widetilde\Gamma_2.$
\end{remark}

Introducing these definitions is justified by the following claim.

\begin{claim}[Informal result]\label{resd-02}
Assume that $R$ is  $m$-stationary. Then,
$$\frac{d^2}{dt^2}H(\mu_t|m)=\left\langle  \Tb _2 \varphi_t+\Tf _2 \psi_t,\mu_t\right\rangle,\quad \forall 0<t<1. $$
\end{claim}

\begin{proof}
Starting from $H(t)= \left\langle \rho\log \rho,m \right\rangle ,$ we obtain $H'(t)=\left\langle 1+\log \rho,\dot \mu \right\rangle=\left\langle \log \rho,\dot \mu \right\rangle =\left\langle -\Ab \varphi+\Af \psi,\mu \right\rangle $ and
 $H''(t)=\frac{d}{dt}\left\langle -\Ab \varphi,\mu \right\rangle +\frac{d}{dt}\left\langle \Af \psi,\mu \right\rangle .$ We have 
\begin{eqnarray*}
\frac{d}{dt}\left\langle -\Ab \varphi,\mu \right\rangle
&=&\left\langle \Ab^2 \varphi -\Ab\dot \varphi-\dot\Ab \varphi,\mu\right\rangle 
=\left\langle \Ab(\Ab-\Bb)\varphi-\left(\frac{\Gb(\dot f,\cdot)}{f}-\frac{\dot f}{f^2}\Gb(f,\cdot)\right)\varphi,\mu \right\rangle \\
&=&\left\langle \Ab(\Ab-\Bb)\varphi-\left(\frac{\Gb(\Lb f,\cdot)}{f}-\frac{\Lb f}{f^2}\Gb(f,\cdot)\right)\varphi,\mu \right\rangle \\
&=&\left\langle \Ab\left(\frac{\Gb(f,\cdot)}{f}-\Cb\right)\varphi-\left(\frac{\Gb(\Lb f,\cdot)}{f}-\Bb \varphi\frac{\Gb(f,\cdot)}{f}\right)\varphi,\mu \right\rangle
=\left\langle \Tb_2 \varphi,\mu \right\rangle .
\end{eqnarray*}
Similarly, we obtain
\begin{eqnarray}\label{eqd-09}
\frac{d}{dt}\left\langle \Af \psi,\mu \right\rangle
&=&\left\langle \Af^2 \psi +\Af\dot \psi+\dot\Af \psi,\mu\right\rangle 
=\left\langle \Af(\Af-\Bf)\psi+\left(\frac{\Gf(\dot g,\cdot)}{g}-\frac{\dot g}{g^2}\Gf(g,\cdot)\right)\psi,\mu \right\rangle \nonumber\\
&=&\left\langle \Af(\Af-\Bf)\psi-\left(\frac{\Gf(\Lf g,\cdot)}{g}-\frac{\Lf g}{g^2}\Gf(g,\cdot)\right)\psi,\mu \right\rangle 
=\left\langle \Tf_2 \psi,\mu \right\rangle ,
\end{eqnarray}
which completes the proof of the claim.
\end{proof}

Gathering Proposition \ref{resd-03} and Claim \ref{resd-02}, we obtain the following
\begin{claim}[Informal result]\label{resd-04} 
When $R$ is $m$-stationary, for all $0<t<1,$
\begin{equation*}
\left\{ \begin{array}{lcc}
\displaystyle{
\frac{d}{dt}H(\mu_t|m)}&=& \left\langle \Tf  \psi_t -\Tb  \varphi_t,\mu_t\right\rangle, \\ \\
\displaystyle{
\frac{d^2}{dt^2}H(\mu_t|m)}&=& \left\langle \Tf _{2} \psi_t+\Tb _{2} \varphi_t,\mu_t \right\rangle .
\end{array}\right.
\end{equation*}
\end{claim}

\section{Brownian diffusion process}\label{sec-diff}

As a first step, we compute informally the operators $ \Theta$ and $ \Theta_2$ associated with the Brownian diffusion process on $\XX=\Rn$ whose  forward and backward generators are given for all $0\le t\le1$ by
\begin{equation}\label{eq-28}
\Lf =\bbf \cdot\nabla + \Delta/2,\qquad \Lb =\bbb \cdot\nabla+ \Delta/2.
\end{equation}
Here, $z\mapsto\bbf(z),\bbb(z)\in\Rn$ are the forward and backward drift vector fields. The term ``informally''  means that we suppose that $\bbf$ and $\bbb$ satisfy some unstated growth and regularity properties which ensure the existence of $R$ and also that Assumptions \ref{ass-03} are satisfied.
\\
Then, we consider reversible Brownian diffusion processes on a compact manifold.

\subsection*{Dynamics of the entropic interpolations}

The associated nonlinear operators  are
$\Bf u=\Delta u/2+\bbf \cdot\nabla u+|\nabla u|^2/2,$ $\Bb u=\Delta u/2+\bbb \cdot\nabla u+|\nabla u|^2/2$  and $\Gf (u,v)=\Gb (u,v)=\nabla u\cdot\nabla v$ for any $t$ and $u,v\in \mathcal{C}^2(\Rn).$
 The expressions 
$$
\left\{ \begin{array}{lcl}
\Af_t&=&\Delta/2+(\bbf +\nabla \psi_t)\cdot\nabla,\\ \Ab_t&=&\Delta/2+(\bbb +\nabla \varphi_t)\cdot\nabla,
\end{array}\right.
$$
of the forward and backward derivatives tell us that the density $ \mu_t(z):= d \mu_t/dz$ solves the following forward-backward system of parabolic PDEs
\begin{equation*}	
\left\{
\begin{array}{ll}
(\partial_t-\Delta/2) \mu_t(z)+\nabla\cdot(\mu_t(\bbf + \nabla \psi_t))(z)=0, &(t,z)\in (0,1]\times\XX\\
\mu_0,	&t=0,
\end{array}\right.
\end{equation*} 
where $\psi$ solves the HJB equation
\begin{equation}\label{eq-18}	
\left\{
\begin{array}{ll}
(\partial_t +\Delta/2+\bbf \cdot\nabla) \psi_t(z)+|\nabla \psi_t(z)|^2/2=0, &(t,z)\in [0,1)\times\XX\\
\psi_1=\log g_1,	&t=1.
\end{array}\right.
\end{equation}
In the reversed sense of time, we obtain
\begin{equation*}	
\left\{
\begin{array}{ll}
(-\partial_t-\Delta/2) \mu_t(z)+\nabla\cdot(\mu_t(\bbb + \nabla\varphi_t))(z)=0, &(t,z)\in [0,1)\times\XX\\
\mu_1,	&t=1,
\end{array}\right.
\end{equation*}
where $\varphi$  solves the HJB equaltion
\begin{equation}\label{eq-18b}	
\left\{
\begin{array}{ll}
(-\partial_t +\Delta/2+\bbb \cdot\nabla) \varphi_t(z)+|\nabla \varphi_t(z)|^2/2=0, &(t,z)\in (0,1]\times\XX\\
\varphi_0=\log f_0,	&t=0.
\end{array}\right.
\end{equation} 
We put $ \log g_1=- \infty$ on the set where $g_1$ vanishes and the boundary condition at time $t=1$ must be understood as $\lim _{ t\uparrow 1} \psi_t= \log g_1.$ Similarly, we have also $\lim _{ t\downarrow 0} \varphi_t= \log f_0$ in $[- \infty,\infty).$

Because of the stochastic representation formulas \eqref{eq-23}, the functions $ \varphi$ and $ \psi$ are the unique viscosity solutions of the above Hamilton-Jacobi-Bellman equations, see \cite[Thm.\,II.5.1]{FS93}. The existence of these solutions is ensured by the Assumptions  \ref{ass-02}.

\begin{remarks}\label{rem-04}\ \begin{enumerate}[(a)]
\item
In  general, $ \varphi$ and $ \psi$  might not  be classical solutions of their respective HJB equations, the gradients $\nabla \psi$ and $\nabla \varphi$ are not defined in the usual sense. One has to consider the notion of  $P$-extended  gradient: $\widetilde\nabla^P,$ that is introduced in \cite{Leo11b}. The complete description of the Markov dynamics of $P$ is a special case of  \cite[Thm.\,5.4]{Leo11b}: the forward and backward drift vector fields of the canonical process under $P$ are respectively  $\bbf+\widetilde\nabla^P \psi_t$ and $\bbb+\widetilde\nabla^P\varphi_t.$

\item
Suppose that the stationary measure $m$ associated with $R$ is equivalent to the Lebesgue measure. Then, the forward and backward drift fields $\bbf$ and $\bbb$ are related to $m$ as follows.
Particularizing the evolution equations with $ \mu_t=m$ and $\nabla \psi=\nabla  \varphi=0,$ we see that the requirement that $R$ is $m$-stationary implies that
\begin{equation}\label{eq-29}
\left\{ \begin{array}{lcl}
\nabla\cdot(m\, \{v ^{ \mathrm{os}}-\nabla \log \sqrt m\})&=&0\\
\nabla\cdot (m\, v ^{ \mathrm{cu}})&=&0
\end{array}\right.
\end{equation}
for all $t$, where $m(x):=dm/dx$, these identities  are considered in the weak sense and 
\begin{equation*}
\left\{ \begin{array}{lcl}
v ^{ \mathrm{os}}&:=&(\bbf+\bbb)/2\\
v ^{ \mathrm{cu}}&:=&(\bbf-\bbb)/2
\end{array}
\right.
\end{equation*}
are respectively the osmotic and the (forward) current velocities of $R$ that were introduced by E.~Nelson in \cite{Nel67}.
In fact, the first equation in \eqref{eq-29} is satisfied in a stronger way, since the duality formula associated to time reversal \cite{Foe86} is
\begin{equation}\label{eq-31}
v ^{ \mathrm{os}}=\nabla \log \sqrt{m}.
\end{equation} 
\end{enumerate}
\end{remarks}

An interesting situation is given by 
$	
\bbf=-\nabla V/2+b _{ \bot}
$	
where $V:\Rn\to\RR$ is $ \mathcal{C}^2$ and the drift vector field $b _{ \bot}$ satisfies
\begin{equation}\label{eq-32}
\nabla\cdot(e ^{ -V} b _{ \bot})=0.
\end{equation}
It is easily seen that  $$m=e ^{ -V} \mathrm{Leb}$$ is the stationary measure of $R.$ Moreover, with  \eqref{eq-31}, we also obtain
\begin{equation}\label{eq-34}
\left\{
\begin{array}{lcl}
\bbf&=&-\nabla V/2+b _{ \bot}\\
\bbb&=&-\nabla V/2-b _{ \bot}
\end{array}\right.,\qquad
\left\{
\begin{array}{lcl}
v ^{ \mathrm{os}}&=&-\nabla V/2\\
v ^{ \mathrm{cu}}&=& b _{ \bot}
\end{array}\right..
\end{equation}
In dimension 2, choosing $b _{ \bot}=e ^{ V}(-\partial_yU,\partial _xU)$ with $U:\RR^2\to\RR$ a $ \mathcal{C}^2$-regular function, solves \eqref{eq-32}. Regardless of the dimension, $b _{ \bot}=e^V S\nabla U$ where $S$ is a constant skew-symmetric matrix and $U:\RR^n\to\RR$ a $ \mathcal{C}^2$-regular function, is also a possible choice. In dimension 3 one can take $b _{ \bot}=e^V\ \nabla \wedge A$ where $A:\RR^3\to \RR^3$ is a $ \mathcal{C}^2$-regular vector field.

\subsection*{Computing $ \Theta$ and $ \Theta_2$}

Our aim is to compute the operators $ \Tf,\Tb,\Tf_2$ and $ \Tb$ of the Brownian diffusion process $R$  with the forward and backward derivatives   given by \eqref{eq-28} and \eqref{eq-34}: 
\begin{equation}\label{eq-46}
\left\{ \begin{array}{lcl}
\Lf&=&(-\nabla V\cdot\nabla+ \Delta)/2+b _{ \bot}\cdot\nabla,\\
\Lb&=&(-\nabla V\cdot\nabla+ \Delta)/2-b _{ \bot}\cdot\nabla,
\end{array}\right. 
\end{equation}
where $\nabla\cdot (e^V b _{ \bot})=0.$ For simplicity, we drop  the arrows and the index $t$ during the intermediate computations.

 \subsubsection*{Computation of $ \Theta$}
 We have 
$
\Gamma(u)= e ^{-u}\Gamma(e^u,u)=|\nabla u|^2
$
and  $Cu=|\nabla u|^2/2.$ Therefore,
$\Theta u=Cu=\Gamma(u)/2=|\nabla u|^2/2.
$
This gives 
\begin{equation}\label{eqd-11a}
\left\{
\begin{array}{lcl}
\Tf_t \psi_t&=&|\nabla \psi_t|^2/2,\\ 
\Tb_t \varphi_t&=&|\nabla \varphi_t|^2/2,
\end{array}
\right.
\end{equation}
and the entropy productions writes as follows, for all $0<t<1,$ 
\begin{equation}\label{eq-22}
\left\{
\begin{array}{lcl}
\If(t)&=& \displaystyle{\ud\IX |\nabla \psi_t|^2 \,d\mu_t},\\ \\
\Ib(t)&=& \displaystyle{\ud\IX |\nabla \varphi_t|^2 \,d\mu_t.}
\end{array}
\right.
\end{equation}
 Recall that $f_t,g_t>0,$ $\mu_t\ae:$ for all $0<t<1.$

\subsubsection*{Computation of $ \Theta_2$}
Since $R$ is a diffusion,  $\Gamma$ is a derivation i.e.
$	
\Gamma(uv,w)=u \Gamma(v,w)+v \Gamma(u,w),
$	
for any regular enough functions $u,v$ and $w$. In particular, the two last terms in the expression of $\Theta_2 u$ simplify:
\begin{equation*}
e ^{-u}\Gamma(e^u Bu,u)-e ^{-u}\Gamma(e^u,u)B u =\Gamma(Bu,u)=\Gamma(Lu,u)+\Gamma(Cu,u),
\end{equation*}
and we get
$ 
\Theta_2 u=L\Gamma(u)/2+\Gamma(Cu,u)-\Gamma (Lu,u)-\Gamma(Cu,u)=L \Gamma(u)/2-\Gamma(Lu,u).
$ 
This means that  
\begin{equation}\label{eq-38}
\left\{
\begin{array}{lcl}
\Tf_2&=& \Gf_2/2,\\
\Tb_2&=&\Gb_2/2.
\end{array}
\right.
\end{equation}
They are precisely half the iterated carré du champ operators $\Gf_2$ and $\Gb_2$ defined at \eqref{eqd-07}. 
The iterated carré du champ $\Gamma_2^o$ of $L^o= \Delta/2$ is\begin{equation*}
\Gamma_2^o(u)=\| \nabla^2 u\|^2_{\mathrm{HS}}=\sum _{ i,j}(\partial _{ ij}u)^2
\end{equation*}
where  $\nabla^2 u$ is the Hessian of $u$ and
 $\| \nabla^2 u\|^2_{\mathrm{HS}}= \mathrm{tr}\, ((\nabla^2u)^2)$ is its squared Hilbert-Schmidt norm. As $\Gamma_2(u)= \Gamma_2^o(u)-2\nabla b(\nabla u)$ where  $2\nabla b(\nabla u)=2\sum _{ i,j}\partial_ib_j\,\partial_iu\partial_ju=[\nabla +\nabla^*] b(\nabla u)$ with $\nabla^* b$ the adjoint of $\nabla b$,
 it follows that
 \begin{equation}\label{eq-17b}
 \left\{
 \begin{array}{lcl}
 \Gf_2(u)&=&\| \nabla^2 u\|^2_{\mathrm{HS}}+(\nabla^2V-[\nabla +\nabla^*]b _{ \bot})(\nabla u),\\
 \Gb_2(u)&=&\| \nabla^2 u\|^2_{\mathrm{HS}}+(\nabla^2V+[\nabla+\nabla^*] b _{ \bot})(\nabla u).
 \end{array}\right.
\end{equation}

\subsubsection*{Regularity problems} 

To make Claim \ref{resd-04} a rigorous statement, one needs to rely upon regularity hypotheses such as   Assumptions \ref{ass-03}. \emph{If one knows that $f_0$, $g_1,$ $\bbf$ and $\bbb$ are such that the linear parabolic equations \eqref{eq-27} admit positive $ \mathcal{C} ^{ 2,2}((0,1)\times\XX)$-regular solutions then we are done}. Verifying this is a  standard (rather difficult) problem which is solved under various hypotheses.  Rather than  giving detail about this PDE problem which can be solved by means of Malliavin calculus, we present  an example  at next subsection.

Working with classical solutions is probably too demanding. The operators $ \Theta$ and $ \Theta_2$ are functions of  $\nabla u$ and the  notion of gradient can be extended as in \cite{Leo11b} in connection with the notion of extended stochastic derivatives, see Remark \ref{rem-04}-(a). Furthermore, when working with integrals with respect to time, for instance when considering integrals of the entropy production (see Section \ref{sec-equilib}), or in situations where $H(t)$ is known to be twice differentiable almost everywhere (e.g.\ $H(t)$ is the sum of a twice differentiable function and a convex function), it would be enough to consider $dt$-almost everywhere defined extended gradients. This program is not initiated in the
present paper.

\subsection*{Reversible Brownian diffusion process on a Riemannian manifold $\XX$}
We give some detail about the standard Bakry-\'Emery setting that already appeared  at Remark  \ref{rem-02}. It corresponds to the case where $b _{ \bot}=0.$
Let us take 
\begin{equation}\label{eqd-12}
m=e ^{-V}\,\vol
\end{equation}
where $V\in \mathcal{C}^2$ satisfies $\IX e ^{-V(x)}\,\vol(dx)<\infty$.  
The $m$-reversible Brownian diffusion process $R\in\MO$ is the Markov  measure   with the initial (reversing) measure $m$ and the  semigroup generator
\begin{equation*}
L=(-\nabla V\cdot \nabla+\Delta)/2.
\end{equation*}
The reversible Brownian motion $R^o\in \MO$ corresponds to $V=0.$ 
Its stochastic derivatives are
\begin{equation}\label{eqd-04}
\Lf^o=\Lb^o=\Delta/2
\end{equation}
and $\vol$ is its reversing measure.
Since  $R^o$ is the unique solution to its own martingale problem, there is a unique $R$ which is both absolutely continuous with respect to   $R^o$ and $m$-reversible.

\begin{lemma}\label{res-03a}
In fact $R$   is specified by
\begin{equation*}
R=\exp \left(-[V(X_0)+V(X_1)]/2+\Iii \left[\Delta V(X_t)/4-|\nabla V(X_t)|^2/8\right]\,dt\right)\, R^o.
\end{equation*}
\end{lemma}

\begin{proof}
To see this, let us define $\hat R$ by means of  $d\hat R/dR^o:=e ^{-V(X_0)}Z_1$ where\\ $Z_t=\exp \left(\int_0^t - \frac{\nabla V}{2} (X_s)\cdot dX_s-\ud \int_0^t| \frac{\nabla V}{2}(X_s)|^2\right)$ is a local positive forward $R^o$-martingale. Since $\IRn e ^{-V(z)}\,dz<\infty$, it follows that $t\mapsto e ^{-V(X_0)}Z_t$ is a forward $R^o$-supermartingale. In particular its expectation $t\mapsto E _{R^o} [e ^{-V(X_0)}Z_t]$ is a decreasing function so that $\hat R(\OO)=E _{R^o}(e ^{-V(X_0)}Z_1)\le E _{R^o}(e ^{-V(X_0)}Z_0)=\IX e ^{-V(z)}\,dz<\infty.$ 
\\
On the other hand,  since both $R^o$ and $d\hat R/dR^o$ are invariant with respect to the  time reversal $X^*$: $X^*_t:=X _{1-t,\ t\in\ii }$, $\hat R$ is also invariant with respect to time reversal: $(X^*)\pf \hat R=\hat R.$ In particular, its endpoint marginals are equal: $\hat R_0=\hat R_1.$ Consequently, $\hat R _{[0,t]}$ doesn't send mass to a cemetery point $\dagger$ outside $\XX$ as time increases, for otherwise its terminal marginal $\hat R_1$ would  give a positive mass to $\dagger,$ in contradiction with $\hat R_0(\dagger)=0$ and $\hat R_0=\hat R_1.$ Hence, $Z$ is a genuine forward $R^o$-martingale. With It\^o's formula, we see that 
 $dZ_t=-Z_t\nabla \frac{V}{2}(X_t)\cdot dX_t,$ $R^o\ae$  and by Girsanov's theory we know that $\hat R$ is a (unique) solution to the martingale problem associated with the generator $L= (-\nabla V\cdot \nabla+\Delta)/2.$ 
Finally, we take $R=\hat R$ and it is easy to check that $L$ is symmetric in $L^2(m),$ which implies that $m$ is reversing. 
\end{proof}

\begin{remark}\label{rem-03}
When $L$ is given by \eqref{eqd-04}, for any non-zero nonnegative  functions $f_0,g_1\in L^2(\vol),$ the smoothing effect of the heat kernels $\rf$ and $\rb$ in \eqref{eq-23}  allows us  to define   \emph{classical} gradients $\nabla \psi_t$ and $\nabla \varphi_t$ for all $t$ in $[0,1)$ and $(0,1]$ respectively. We see that $\nabla\psi_t$ and $\nabla\varphi_t$ are  the forward and backward drift vector fields of the canonical process under $P.$ 
\end{remark}

Next result proposes a general setting where  Assumptions \ref{ass-03} are satisfied. The manifold $\XX$ is assumed to be compact to avoid integrability troubles.

\begin{proposition}\label{res-03}
Suppose that $\XX$ is a compact  Riemannian manifold without boundary and $V:\XX\to\RR$  is $\mathcal{C}^4$-regular.  Then, for any $ \mathcal{C}^2$-regular function $u:\XX\to\RR,$ the function $(t,x)\mapsto u(t,x):=E _{ R}[u(X_t)\mid X_0=x]$ belongs to $ \mathcal{C}^{ 1,2}((0,1)\times\XX).$

In particular, if in addition $\XX$ is assumed to be connected and $f_0$ and $g_1$ are non-zero nonnegative  $ \mathcal{C}^2$-regular functions, the functions $f$ and $g$ defined at \eqref{eq-23} and their logarithms $ \varphi$ and $ \psi$ are classical solutions of \eqref{eq-27} and \eqref{eqd-02}.
\end{proposition} 

\begin{proof}
For any path $ \omega\in\OO$ and any $0\le t\le1,$ we denote 
$$
Z_t( \omega):= \exp \left(-[V(\omega_0)+V(\omega_t)]/2+\int_0^t \left[\Delta V(\omega_s)/4-|\nabla V(\omega_s)|^2/8\right]\,ds\right).
$$
With Lemma \ref{res-03a} we see that for 	all $0\le t\le 1,$  $Z_t= \frac{dR _{ [0,t]}}{dR^o _{ [0,t]}}$ and
\begin{equation*}
u(t,x)= \frac{E _{ R^o}[ u(X_t)Z_1\mid X_0=x]}{E _{ R^o}[Z _1\mid X_0=x]}
=e ^{ V(x)} \EE[u(B_t ^x)Z_t(B^x)]
\end{equation*}
where $(B _{ s}^x) _{ 0\le s\le 1}$ is a Brownian motion starting from $x$ under some abstract probability measure  whose expectation  is denoted by $\EE$. By means of parallel transport, it is possible to build on any small enough neighborhood $U$ of $x$ a coupling such that $x'\in U\mapsto B ^{ x'}$ is almost surely continuous with respect to the uniform topology on $\OO.$ This coupling corresponds to $B^x=x+B^0$ in the Euclidean case.  The announced  $x$-regularity is a consequence of our assumptions which allow us to differentiate (in the usual deterministic sense) in the variable $x$ under the expectation sign $\EE.$  On the other hand, the  $t$-regularity is a consequence of stochastic differentiation: apply Itô's formula to $u(B^x_t)$ and take advantage of the martingale property of $Z_t(B^x)$.

As regards last statement, the connectivity assumption implies the positivity of $f$ and $g$ on $(0,1)\times \XX$  as soon as $f_0$ and $g_1$ are  nonnegative and not vanishing everywhere.
 \end{proof}

We gave the detail
 of the proof of Proposition \ref{res-03} because of its ease. But in view of Remark \ref{rem-03}, the requirement that $f_0$ and $g_1$ are $ \mathcal{C}^2$ is not optimal.  However, this restriction will not be harmful when investigating convexity properties of the entropy along interpolations.

In view of Proposition \ref{res-03}, we remark that under its assumptions the functions $f_t,g_t$ belong to the domain of the carré du champ operator. Hence, for any $0<t<1,$ the stochastic derivatives of $P$ are well-defined on $ \mathcal{C}^2(\XX)$  and equal to
$$
\left\{ \begin{array}{lcl}
\Af_t&=&\Delta/2+\nabla(-V/2+ \psi_t)\cdot\nabla,\\ \Ab_t&=&\Delta/2+\nabla (-V/2+\varphi_t)\cdot\nabla.
\end{array}\right.
$$
 Here,  $\psi$ and $ \varphi$ are respectively the classical solutions (compare Remark \ref{rem-04}-(a)) of the HJB equations \eqref{eq-18} and \eqref{eq-18b} with $\bbf=\bbb=-\nabla V/2.$
\\
Bochner's formula relates the iterated carré du champ $\Gamma_2^o$ of $L^o= \Delta/2$ and the Ricci curvature:
\begin{equation}\label{eq-50}
\Gamma_2^o(u)=\| \nabla^2 u\|^2_{\mathrm{HS}}+\mathrm{Ric}(\nabla u)
\end{equation}
where  $\nabla^2 u$ is the Hessian of $u$ and
 $\| \nabla^2 u\|^2_{\mathrm{HS}}= \mathrm{tr}\, ((\nabla^2u)^2)$ is its squared Hilbert-Schmidt norm. As $\Gamma_2(u)= \Gamma_2^o(u)-[\nabla+\nabla^*] b(\nabla u)$, it follows that
 \begin{equation}\label{eq-17}
\Gamma_2(u)=\| \nabla^2 u\|^2_{\mathrm{HS}}+(\mathrm{Ric}+\nabla^2 V)(\nabla u).
\end{equation}

\begin{theorem}[Reversible Brownian diffusion process]\label{resd-06}
Let the reference Markov measure $R$ be associated with $L$ given at \eqref{eq-11} on 
 a compact connected  Riemannian manifold $\XX$  without boundary. It is  assumed that $V:\XX\to\RR$  is $\mathcal{C}^4$-regular and $f_0$, $g_1$ are non-zero nonnegative  $ \mathcal{C}^2$-regular functions. Then,   along the entropic interpolation $[\mu_0,\mu_1]^R$ associated with $R$, $f_0$ and $g_1$, we have for all $0<t<1,$
\begin{eqnarray*}
\frac{d}{dt}H(\mu_t|e ^{ -V}\vol)&=&\left\langle \ud\Big(|\nabla\psi_t|^2-|\nabla\varphi_t|^2\Big),\mu_t \right\rangle,\\
\frac{d^2}{dt^2}H(\mu_t|e ^{ -V}\vol)&=&\left\langle \ud\Big(\Gamma_2(\psi_t)+\Gamma_2(\varphi_t)\Big),\mu_t \right\rangle
\end{eqnarray*} 
where  $ \psi$ and $ \varphi$ are the classical solutions of the HJB equations \eqref{eq-18} and \eqref{eq-18b} with $\bbf=\bbb=-\nabla V/2$,  $\Gamma(u)=|\nabla u|^2$ and $ \Gamma_2$ is given by \eqref{eq-17}.
\end{theorem}

\begin{proof}
The assumptions imply that $f_0$ and $g_1$ satisfy     \eqref{eq-13} and they allow us to apply Proposition \ref{res-03} so that Assumptions \ref{ass-03} are satisfied and Claim \ref{resd-04} is a rigorous result.
\end{proof}

\section{Random walk on a graph}\label{sec-RW}

Now we take as our reference measure $R$ a continuous-time Markov process  on  a countable state space $\XX$ with a graph structure $(\XX,\sim)$. The set $\XX$ of all vertices is equipped with the graph relation  $x\sim y$ which  signifies that $x$ and $y$ are adjacent, i.e.\ $\left\{x,y\right\}$ is a non-directed edge. The degree $n_x:=\#\left\{y\in\XX,x\sim y\right\}$ of each $x\in\XX$ is the number of its neighbors.  
It is assumed that $(\XX,\sim)$ is locally finite, i.e.\  $n_x<\infty$ for all $x\in\XX$ and also that 
$(\XX,\sim)$ is connected. This means that for any couple $(x,y)\in\XXX$ of different states, there exists a finite chain $(z_1,\dots, z_k)$ in $\XX$ such that $x\sim z_1\sim z_2\dots \sim z _{k}\sim y$. In particular $n_x\ge1,$ for all $x\in\XX.$

\subsection*{Dynamics of the entropic interpolation}

A (time homogeneous) random walk on $(\XX,\sim)$ is a Markov measure with forward derivative $\partial+\Lf$ defined for all function $u\in \RR ^{\XX}$ by
\begin{equation*}
\Lf  u(x)=\sum _{y :x\sim y} (u_y -u_x)\Jf_x (y)=:\IX D_x u\,d\Jf _{ x},\qquad 0\le t<1,x\in\XX, 
\end{equation*}
where $\Jf_x (y)$ is the  instantaneous frequency of forward jumps from $x$ to $y$,
$$
D_xu(y)=Du(x,y):=u(y)-u(x)
$$
is the discrete gradient and 
\begin{equation*}
\Jf _{ x}=\sum _{y :x\sim y} \Jf_x(y)\, \delta_y \in\MX,
\qquad 0\le t<1,\ x\in\XX
\end{equation*}
is the  forward jump kernel. Similarly, one denotes its backward derivative by
 $$
 \Lb u(x)=\IX D_xu\,d\Jb _{x},\qquad x\in\XX, u\in \RR ^{\XX}.
 $$
For the  $m$-stationary Markov measure $R$ with stochastic derivatives $\Lf$ and $\Lb$,  the time-reversal duality formula is
$$m(x)\Jf_x(y)=m(y)\Jb_y(x),\quad\forall x\sim y\in\XX.$$
 The expressions  $\Lf u$ and $\Lb u$ can be seen as  the matrices
\begin{equation}\label{eqd-17}
\left\{ \begin{array}{lcl}
\Lf =\big(\1 _{\left\{x\sim y\right\} } \Jf_x(y)-\1_{\{x=y\}}\Jf _{ x}(\XX)\big)_{x,y\in\XX}\\
\Lb =\big(\1 _{\left\{x\sim y\right\} } \Jb_x(y)-\1_{\{x=y\}}\Jb _{ x}(\XX)\big)_{x,y\in\XX}
\end{array}\right.
\end{equation}
 acting on the column vector $u=[u_x] _{ x\in\XX}.$ Therefore, the solutions  $f$ and $g$ of \eqref{eq-27} are
\begin{equation*}
f(t)= e ^{ t\Lb} \, f_0,\qquad
g(t)=e ^{( t-1)\Lf} \, g_1,
\end{equation*}
where $f:t\in\ii\mapsto [f_x]_{x\in\XX}(t)\in \RR^\XX$ and $g:t\in\ii\mapsto [g_x]_{x\in\XX}(t)\in \RR^\XX$ are column vectors, whenever these exponential matrices are well-defined. 

Let us  compute the operators $B$, $ \Gamma$ and $A.$ By a direct computation, we obtain for all $0<t<1$ and $x\in\XX,$
\begin{equation*}
\left\{ \begin{array}{lcl}
\Bf  u(x)&=&\IX (e ^{ D_xu}-1)\, d\Jf _{ x},\\
\Bb  u(x)&=&\IX (e ^{ D_xu}-1)\, d\Jb _{ x},
\end{array}\right.
\quad
\left\{ \begin{array}{lcl}
\Gf (u,v)(x)&=&\IX D_xu D_xv\, d \Jf _{ x},\\
\Gb (u,v)(x)&=&\IX D_xu D_xv\, d \Jb _{ x},
\end{array}\right.
\end{equation*}
and
\begin{equation*}
\left\{ \begin{array}{lclcl}
\Af _{ t}u(x)&=&\IX D_xu\, e ^{ D _{ x} \psi_t}\, d\Jf _{ x}&=&\displaystyle{\sum _{ y:x\sim y} [u(y)-u(x)]\, \frac{g_t(y)}{g_t(x)}\, \Jf _{ x}(y)},\\ \\
\Ab _{ t}u(x)&=&\IX D_xu\, e ^{ D _{ x} \varphi_t}\, d\Jb _{ x}&=&\displaystyle{\sum _{ y:x\sim y} [u(y)-u(x)]\, \frac{f_t(y)}{f_t(x)}\, \Jb _{ x}(y)}.

\end{array}\right.
\end{equation*}
The matrix representation of these operators is
\begin{equation*}
\left\{ \begin{array}{lcl}
\Af_t=\Af ({g_t})=\left(\1 _{\left\{x\sim y\right\} }\displaystyle{\frac{g_t(y)}{g_t(x)}} \Jf_x(y)-\1_{\{x=y\}}\sum _{ z:z\sim x} \displaystyle{\frac{g_t(z)}{g_t(x)}} \Jf_x(z)\right)_{x,y\in\XX}\\ \\
\Ab_t=\Ab ({f_t})=\left(\1 _{\left\{x\sim y\right\} }\displaystyle{\frac{f_t(y)}{f_t(x)}} \Jb_x(y)-\1_{\{x=y\}}\sum _{ z:z\sim x} \displaystyle{\frac{f_t(z)}{f_t(x)}} \Jb_x(z)\right)_{x,y\in\XX}\end{array}\right.
\end{equation*}
The forward-backward systems describing the evolution  of $[ \mu_0,\mu_1]^R$ are

\begin{equation*}
\mu_t= {\mu_0\, \exp \left( \int_0^t \Af({e ^{ (s-1)\Lf }\, g_1})\,ds\right)}= {\mu_1\, \exp \left( \int_t^1 \Ab(e ^{ s\Lb}f_0)\,ds\right)},\quad 0\le t\le1
\end{equation*}
where the measures $ \mu_t$ are seen as row vectors and the functions $f_0,g_1$ as column vectors.

\subsection*{Derivatives of the entropy}

A pleasant feature of the discrete setting is that no spatial   regularity is required for a function to be in the domain of the operators $L,\Gamma,\dots$ The only required regularity is that the functions $f$ and $g$ defined by \eqref{eq-23} have to be twice differentiable with respect to  $t$  on the open interval $(0,1).$ But this is ensured by the following lemma.

\begin{lemma}\label{res-04}
 Consider $f_0, g_1$ as in Definition \ref{def-fg} and suppose that $f_0,g_1\in L^1(m)\cap L^2(m)$ where the stationary measure $m$ charges every point of $\XX.$ 
Then, under the  assumption that
\begin{equation}\label{eq-16}
\sup _{ x\in\XX} \{\Jf _{ x}(\XX)+\Jb_x(\XX)\}< \infty,
\end{equation}
 for each $x\in\XX,$ $t\mapsto f(t,x)$ and $t\mapsto g(t,x)$ are  $ \mathcal{C}^{ \infty}$-regular on $(0,1)$.
\end{lemma}

\begin{proof}
For any $0\le t\le 1,$ $f_t$ and $g_t$ are well-defined $P_t$\ae, where $P$ is given by \eqref{eq-15}. But, as $R$ is an irreducible random walk, for each $0<t<1,$ ``$P_t$\ae''  is equivalent to ``everywhere''. Since $f_0\in L^1(m)$ and $R$ is $m$-stationary,  we have $\IX f_t\, dm=E_RE_R[f_0(X_0)\mid X_t]=E_R(f_0(X_0))=\IX f_0\,dm< \infty,$ implying that $f_t\in L^1(m).$ As  $\sup_x\Jb_x(\XX)< \infty$,  $\Lb$ is a bounded  operator on $L^1(m)$. This implies that $t\in (0,1)\mapsto f_t= e ^{ t\Lb}f_0\in L^1(m)$ is differentiable and $(d/dt)^k f_t=\Lb^k f_t$ for any $k$. As $m$ charges every point, we also see that $t\in (0,1)\mapsto f_t(x)\in \RR$ is infinitely differentiable for every $x$. A similar proof works with $g$ instead of $f$.
\end{proof}

As an important consequence of Lemma \ref{res-04} for our purpose, we see that {the statement of Claim \ref{resd-04} is rigorous} in the present discrete setting.

\begin{theorem}\label{resd-08} 
Let $R$ be an $m$-stationary random walk with  jump measures $\Jf$ and $\Jb$ which satisfy \eqref{eq-16}.  Along any entropic interpolation $[ \mu_0,\mu_1]^R$ associated with a couple
 $(f_0, g_1)$ as in Definition \ref{def-fg} and such that $f_0,g_1\in L^1(m)\cap L^2(m)$, we have for all $0<t<1,$
\begin{equation*}
\frac{d}{dt}H(\mu_t|m)= \sum _{ x\in\XX} \left(\Tf\psi_t -\Tb \varphi_t\right) (x)\,\mu_t(x), 
\
\frac{d^2}{dt^2}H(\mu_t|m)= \sum _{ x\in\XX}\left(  \Tf _2 \psi_t+\Tb _2 \varphi_t\right)(x)\,\mu_t(x)
\end{equation*}
where the expressions of $ \Tf\psi_t$ and $\Tb \varphi_t$ are given at \eqref{eqd-11} and the expressions of $ \Tf_2\psi_t$ and $\Tb _2\varphi_t$ are given at Proposition \ref{res-08} below.
\end{theorem}

\subsection*{Computing $ \Theta$ and $ \Theta_2$ }

Our aim now is to compute 
the operators $\Tf, \Tb, \Tf_2$ and $\Tb_2$ for a general random walk $R$. 
We  use the shorthand notation $\IX a(x,y)\,J_x(dy)=[\IX a\,dJ](x)$ and drop the arrows and the index $t$ during the intermediate computations.
The  Hamilton-Jacobi operator is 
$
Bu:=e ^{-u}Le^u=\IX (e ^{Du}-1)\,dJ
$
which gives
$$	
Cu:=(B-L)u=\IX \theta(Du)\,dJ
$$	
where the function $\theta$ is defined by
\begin{equation*}
\theta(a):=e^a-a-1,\quad a\in\RR.
\end{equation*}
Compare  $Cu=|\nabla u|^2/2,$ remarking that $\theta(a)=a^2/2+o _{a\to 0}(a).$
The convex conjugate $\ts$ of $\theta$ will be used in a moment. It is given by
\begin{equation*}
\ts(b)= \left\{\begin{array}{ll}
(b+1)\log(b+1)-b,& b>-1,\\
1,&b=-1,\\
\infty,&b<-1.
\end{array}\right.
\end{equation*}

\subsubsection*{Computation of $ \Theta$}

The carr\'e du champ is
$
\Gamma(u,v)= \IX DuDv\,dJ
$
so that
$
e ^{-u}\Gamma(e^u,u)= \IX Du(e ^{Du}-1)\,dJ.
$
Since $a(e^a-1)-\theta(a)=ae^a-e^a+1=\ts(e^a-1),$ with $\Theta u:=e ^{-u}\Gamma(e^u,u)-Cu,$ we obtain    
\begin{equation}\label{eqd-11}
\left\{
\begin{array}{lcl}
\Tf \psi_t(x)&=&\displaystyle{\sum _{ y:x\sim y} \ts\left(\frac{Dg_t(x,y)}{g_t(x)}\right)\,\Jf _{ x}(y)},\\
\Tb \varphi_t(x)&=&\displaystyle{\sum _{ y:x\sim y} \ts\left(\frac{Df_t(x,y)}{f_t(x)}\right)\,\Jb _{ x}(y)}.
\end{array}\right.
\end{equation}
where we used
$e ^{ D \psi_t(x,y)}-1= Dg_t(x,y)/g_t(x)$ and   $e^{ D \varphi_t(x,y)}-1= Df_t(x,y)/f_t(x).$ The ratio $Dg/g$ should be seen as the discrete logarithmic derivative of $g$. Recall that $g_t>0,$ $\mu_t\ae:$ we do not divide by zero.  Compare  $\Theta \psi=|\nabla  g/g|^2/2,$ remarking that $\ts(b)=b^2/2+o _{b\to 0}(b).$

\subsubsection*{Entropy production}

The entropy productions are
\begin{equation}\label{eqd-14}
\left\{
\begin{array}{lcl}
\If(t)&=&\sum _{ (x,y): x\sim y} \ts \left(\frac{Dg_t(x,y)}{g_t(x)}\right) \,\mu_t(x)
\Jf _{ x}(y),\\
\Ib(t)&=&\sum _{ (x,y): x\sim y} \ts \left(\frac{Df_t(x,y)}{f_t(x)}\right) \, \mu_t(x)\Jb _{ x}(y).
\end{array}\right.
\end{equation}

\subsubsection*{Computation of $\Theta_2$}

Unlike the diffusion case, no welcome cancellations occur. The carr\'e du champ $\Gamma$ is not a derivation anymore since $\Gamma(uv,w)-[u \Gamma(v,w)+v \Gamma(u,w)]=uvLw+\IX DuDvDw\,dJ$. 
Furthermore, we also loose the simplifying identity $C=\Theta.$ To give a readable expression of $\Theta_2$, it is necessary to introduce some simplifying shorthand notation:
$$
\left\{\begin{array}{lcl}
\aa&=&D u(x,y)\\
\aa'&=&D u(x,y')\\
\bb&=&D u(y,z)\\
\cc&=&D u(x,z)
\end{array}\right.,
\quad
\left\{\begin{array}{lcl}
\Ixy F(\aa)&=&\IX F(D u(x,y))\,J_x(dy)\\
\Ixyy F(\aa,\aa')&=&\IXX F(D u(x,y), D u(x,y'))\,J_x(dy)J_x(dy')\\
\Ixyz F(\aa,\bb)&=&\IXX F(D u(x,y), D u(y,z))\,J_x(dy)J_y(dz)\\
\Ixyz F(\aa,\cc)&=&\IXX F(D u(x,y),D u(x,z))\,J_x(dy)J_y(dz).\\
\end{array}\right. 
$$
Remark that $\Ixyz F(\cc)=\IX F(D u(x,z))\,J_x^2(dz)$ with $J_x^2(dz):=\int _{y\in\XX}J_x(dy)J_y(dz)$ and $\aa+\bb=\cc.$
\bigskip
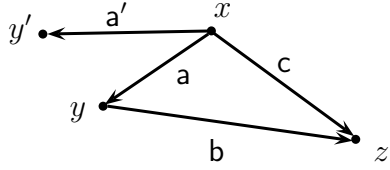
\begin{center}
\scalebox{1} 
{
\begin{pspicture}(0,-1.15)(5.4,1.15)
\psdots[dotsize=0.12](2.84,0.67)
\psdots[dotsize=0.12](1.4,-0.33)
\psdots[dotsize=0.12](4.76,-0.77)
\psline[linewidth=0.04cm,arrowsize=0.05291667cm 3.0,arrowlength=1.4,arrowinset=0.4]{->}(2.86,0.65)(4.74,-0.73)
\psline[linewidth=0.04cm,arrowsize=0.05291667cm 3.0,arrowlength=1.4,arrowinset=0.4]{->}(2.8,0.63)(1.42,-0.31)
\psline[linewidth=0.04cm,arrowsize=0.05291667cm 3.0,arrowlength=1.4,arrowinset=0.4]{->}(1.46,-0.33)(4.72,-0.79)
\psline[linewidth=0.04cm,arrowsize=0.05291667cm 3.0,arrowlength=1.4,arrowinset=0.4]{->}(2.8,0.67)(0.66,0.63)
\usefont{T1}{ptm}{b}{n}
\rput(2.99,0.955){$x$}
\usefont{T1}{ptm}{m}{n}
\rput(1.07,-0.405){$y$}
\usefont{T1}{ptm}{m}{n}
\rput(5.1,-0.965){$z$}
\usefont{T1}{ptm}{m}{n}
\rput(0.3,0.695){$y'$}
\usefont{T1}{ptm}{m}{n}
\rput(2.47,0.055){$\aa$}
\usefont{T1}{ptm}{m}{n}
\rput(2.91,-0.905){$\bb$}
\usefont{T1}{ptm}{m}{n}
\rput(3.81,0.215){$\cc$}
\usefont{T1}{ptm}{m}{n}
\rput(1.58,0.895){$\aa'$}
\psdots[dotsize=0.12](0.6,0.63)
\end{pspicture} 
}
\end{center}
We also define the function 
\begin{equation}\label{eq-37}
h(a):=\ts(e^a-1)=ae^a-e^a+1,\quad a\in\RR.
\end{equation}

\begin{lemma}\label{res-07}
For any function $u\in \RR^\XX$ and any $x\in\XX,$
\begin{eqnarray*}
\Theta u(x)&=&\Ixy h(\aa),\\
\Theta_2 u(x)&=&\Big(\Ixy (e^\aa -1)\Big)^2+
	\Ixy[J_y(\XX)-J_x(\XX)] h(\aa )+\Ixyz [2e^\aa h(\bb )-h(\cc )].
\end{eqnarray*}
\end{lemma}

\begin{proof}
The first identity is \eqref{eqd-11}. Let us look at $ \Theta_2u(x).$
Rather than formula \eqref{eqd-10}, for an explicit formulation of $\Theta_2$ in terms of $J$, it will be easier to go back to \eqref{eqd-09}:
$	
\Tf_2 \psi_t=\Af_t^2 \psi_t+\frac{d}{dt}(\Af_t \psi_t)
$	
with $\Af_tu=\IX Du\, e ^{D \psi_t}\,d\Jf$ and $\dot \psi=-\Bf \psi=-\IX (e ^{D \psi}-1)\,d\Jf.$ 
\\
Making use of $A \psi(x)=\Ixy \aa e^\aa$ and $\dot \psi(x)=\Ixy -(e^\aa-1),$ we see that
\begin{eqnarray*}
A^2 \psi(x)&=& \IX (A \psi_y-A \psi_x)e ^{D \psi(x,y)}\,J_x(dy)
	=\Ixyz e^\aa\bb e^\bb-\Ixyy e^\aa\aa' e ^{\aa'}	     \\
 \frac{d}{dt}(A \psi)(x)&=&\IX (D \psi(x,y) e ^{D \psi(x,y)}+e ^{D \psi(x,y)})(\dot \psi_y-\dot \psi_x)\,J_x(dy)\\
	&=& -\Ixyz (\aa e^\aa+\aa)(e^\bb-1)+\Ixyy (\aa e^\aa+\aa)(e ^{\aa'}-1)
\end{eqnarray*}
where $\aa,\aa'$ and $\bb$ are taken with $u=\psi.$
This shows that
\begin{equation*}
\Theta _2 u(x)=\Ixyz e^\aa \bb e^\bb -(\aa e^\aa +e^\aa )(e^\bb -1)
	-\Ixyy e^\aa \aa 'e^{\aa '}-(\aa e^\aa +e^\aa )(e^{\aa '}-1)
\end{equation*}
The functions defining the integrands  rewrite as follows
\begin{eqnarray*}
e^abe^b-(ae^a+e^a)(e^b-1)
	&=& h(a)+2e^ah(b)-h(c)\quad \textrm{with}\quad c=a+b,\\
e^aa'e^{a'}-(ae^a+e^a)(e^{a'}-1)
	&=& (e^a-1)h(a')-h(a)(e ^{a'}-1)+h(a')-(e^a-1)(e ^{a'}-1).
\end{eqnarray*}
Hence, 
\begin{eqnarray*}
&&\Ixyy e^\aa \aa 'e^{\aa '}-(\aa e^\aa +e^\aa )(e^{\aa '}-1)\\
	&=& \Ixyy(e^\aa -1)h(\aa ')-\Ixyy h(\aa )(e ^{\aa '}-1)+\Ixyy h(\aa ')-\Ixyy (e^\aa -1)(e ^{\aa '}-1)
\\
	&=&J_x(\XX)\Ixy h(\aa )-\left(\Ixy (e^\aa -1)\right)^2
\end{eqnarray*}
and  the desired result follows.
\end{proof}

Lemma \ref{res-07} leads us to the following evaluations.

\begin{proposition}\label{res-08}
 Consider $f_0, g_1$ as in Definition \ref{def-fg} and suppose that $f_0,g_1\in L^1(m)\cap L^2(m)$ where the stationary measure $m$ charges every point of $\XX.$ Also assume that $\Jf,\Jb$ satisfy \eqref{eq-16}. Then,
\begin{eqnarray}\label{eq-26a}
&&\Tf _{ 2} \psi_t(x)\\&&\quad=
	\left( \sum _{ y:x\sim y} (e ^{ D\psi_t(x,y)}-1)\, \Jf _x (y)\right)^2
	+ \sum _{ y:x\sim y}[\Jf _y (\XX)-\Jf _x (\XX)] \theta^* \left( e ^{ D\psi_t(x,y)}-1\right) \Jf _x (y)\notag\\
	&&\qquad +\sum _{ (y,z):x\sim y\sim z} \left[2 \theta^* \left( e ^{ D\psi_t(y,z)}-1\right) \Af _{ t,x} (y)\Jf _y (z)
	- \theta^* \left( e ^{ D\psi_t(x,z)}-1\right) \,\Jf _x (y)\Jf _y (z)  \right]\notag
\end{eqnarray}
and
\begin{eqnarray}\label{eq-26b}
&&\Tb _{ 2} \varphi_t(x)\\&&\quad=
	\left( \sum _{ y:x\sim y} (e ^{ D\varphi_t(x,y)}-1) \Jb _x (y)\right)^2
	+ \sum _{ y:x\sim y}[\Jb _y (\XX)-\Jb _x (\XX)] \theta^* \left( e ^{ D\varphi_t(x,y)}-1\right) \Jb _x (y)\notag\\
	&&\qquad +\sum _{ (y,z):x\sim y\sim z} \left[2 \theta^* \left( e ^{ D\varphi_t(y,z)}-1\right) \Ab _{ t,x} (y)\Jb _y (z)
	- \theta^* \left( e ^{ D\varphi_t(x,z)}-1\right) \,\Jb _x (y)\Jb _y (z)  \right]\notag
\end{eqnarray}
where
\begin{equation*}
\Af _{ t,x} (y)=e ^{ D\psi_t(x,y)}\,\Jf_x(y)=\displaystyle{\frac{g_t(y)}{g_t(x)}\,\Jf _x (y)},\qquad
\Ab _{ t,x} (y)=e^{ D\varphi_t(x,y)}\,\Jb_x(y)=\displaystyle{\frac{f_t(y)}{f_t(x)}\,\Jb_x (y)}
\end{equation*}
are the forward and backward jump frequencies of $P=f_0(X_0)g_1(X_1)\,R.$
\end{proposition}

 \subsection*{Reversible random walks}
 
 Let us take a positive measure  $m=\sum _{ x\in\XX}m_x\, \delta_x\in\MX$ with $m_x>0$ for all $x\in\XX.$ It is easily checked with 
 the detailed balance condition
\begin{equation}\label{eq-07}
m(dx)J_x(dy )=m(dy )J_{y }(dx)
\end{equation}
which characterizes the reversibility of $m$
 that the jump kernel 
\begin{equation}\label{eq-14}
J_x=\sum _{y :x\sim y} s(x,y) \sqrt{m_y/m_x}  \ \delta_y 
\end{equation}
where $s(x,y)=s(y,x)>0$ for all $x\sim y,$ admits $m$ as a reversing measure. As $(\XX,\sim)$ is assumed to be connected, the random walk is irreducible and the reversing measure  is unique  up to a multiplicative constant. 
 
\begin{examples} \label{ex-01}
Let us present the simplest  examples of reversible random walks.
\begin{enumerate}[(a)]
\item
\emph{Reversible counting random walk.} \ 
  The simplest example is provided by the counting  jump kernel 
$     
J_x=\sum _{y :x\sim y} \delta_y \in\PX,
$
$ 
x\in\XX
$	
which admits the counting measure
$	
m=\sum_x \delta_x
$	
as a reversing measure: take $m_x=1$ and $s(x,y)=1$ in \eqref{eq-14}.

\item
\emph{Reversible simple random walk.}\ The simple jump kernel is
$	
J_x=\frac{1}{n_x}\sum _{y :x\sim y} \delta_y \in\PX,
$
$
x\in\XX.
$	
The measure
$	
m=\sum_x n_x \delta_x
$	
is a reversing measure: take  $m_x=n_x$ and $s(x,y)=(n_xn_y)^{ -1/2}$ in \eqref{eq-14}. 
 \end{enumerate}
\end{examples}

The dynamics of $R$ is as follows. Once the walker is at $x$, one starts an exponential random clock with  frequency $J_x(\XX)$. When the clock rings, the walker jumps at random onto a neighbor of $x$ according to the probability law $J_x(\XX) ^{ -1}\sum _{ y:x\sim y}J(x;y)\, \delta _{ y}.$ This procedure goes on and on.
Since $R$ is reversible, we have $\Jf=\Jb=J$ and  the forward and backward frequencies of jumps of the $(f,g)$-transform $P=f_0(X_0)g_1(X_1)\,R$ are respectively 
$$
\Af _x (y)= \frac{g_t(y)}{g_t(x)}J_x(y)\quad \textrm{and}\quad\Ab _x (y)= \frac{f_t(y)}{f_t(x)}J_x(y).
$$

As a direct consequence of Theorem \ref{resd-08} and Proposition \ref{res-08}, we obtain the following result.

\begin{theorem}\label{res-09}
Let $R$ be an $m$-reversible random walk with  jump measure $J$ which is given by \eqref{eq-14} and satisfies
$$\sup _{ x\in\XX}J_x(\XX)< \infty.$$ 
For this condition to be satisfied, it suffices that 
for some $0< c, \sigma< \infty$, we have for all $x\sim y\in\XX$,
\begin{equation}\label{eq-21}
\left\{
\begin{array}{l}
m_y/n_y\le c m_x/n_x,\\
0<s(x,y) \sqrt{n_xn_y}\le \sigma,
\end{array}\right.
\end{equation}
with the notation of \eqref{eq-14}.
\\
Then,
along any entropic interpolation $[ \mu_0,\mu_1]^R$ associated with a couple
 $(f_0, g_1)$ as in Definition \ref{def-fg} and such that $f_0,g_1\in L^1(m)\cap L^2(m)$, we have for all $0<t<1,$
\begin{equation*}
\frac{d}{dt}H(\mu_t|m)= \sum _{ x\in\XX} [\Theta\psi_t - \Theta \varphi_t](x)\,\mu_t(x), 
\qquad
\frac{d^2}{dt^2}H(\mu_t|m)= \sum _{ x\in\XX} [\Theta _2 \psi_t+ \Theta _2 \varphi_t](x)\,\mu_t(x) 
\end{equation*}
where the expressions of $ \Theta\psi_t$ and $ \Theta \varphi_t$ are given at \eqref{eqd-11} and the expressions of $ \Theta_2\psi_t$ and $ \Theta _2\varphi_t$ are given at Proposition \ref{res-08} (drop the useless time arrows).
\end{theorem}

\section{Convergence to equilibrium}\label{sec-equilib}

Let $R$ be an $m$-stationary Markov probability measure. In this section $m$ is assumed to  be a probability measure (and so is $R$). We are interested in the convergence as $t$ tends to infinity of the forward heat flow
\begin{equation}\label{eq-43}
\mu_t:=(X_t)\pf (\rho_0(X_0)\,R)\in\PX,\quad t\ge 0,
\end{equation}
where $ \rho_0=d \mu_0/dm$ is the initial density.
If the solution of the forward heat equation $ \left\{ \begin{array}{l}
\partial_t \nu_t=\nu_t\Lf\\
\nu_0= \nu^o
\end{array}\right. $  with the initial profile  $ \nu^o= \mu_0$ is unique, then it is the heat flow \eqref{eq-43}. But in general $(\mu_t) _{ t\ge0}$ is only \emph{a} solution of this heat equation.
We are going to prove by  implementing the ``stochastic process strategy'' of the present paper, that under some hypotheses, the following convergence to equilibrium $$\Lim t \mu_t=m$$ holds.   As usual,  the function $t\mapsto H(\mu_t|m)$ is an efficient Lyapunov function of the dynamics.

\subsection*{Stationary dynamics}

Let us first make the basic assumptions precise.

\begin{assumptions}\label{ass-05}
The stationary measure $m$ is a probability measure and we assume the following.
\begin{enumerate}[(a)]
\item
Brownian diffusion. The forward derivative is $$\Lf=\bbf\cdot\nabla+ \Delta/2$$  on a connected Riemannian manifold $\XX$ without boundary. The initial density $\rho_0:=d\mu_0/dm$  is regular enough for $ \rho_t(z)=d \mu_t/dm(z)$ to be positive and twice differentiable in $t$ and $z$ on $(0,\infty)\times\XX$.
\item
Random walk. The forward derivative is $$\Lf u(x)=\sum _{ y:x\sim y}[u(y)-u(x)]\,\Jf_x(y)$$ on the countable locally finite connected graph $(\XX,\sim)$. We assume that \eqref{eq-16} holds:  $\sup_x\{\Jf_x(\XX)+\Jb_x(\XX)\}< \infty.$
\end{enumerate}
\end{assumptions}

The point in the following Theorem \ref{resd-07} is that no curvature bound such as \eqref{eq-35} below nor reversibility  are required.
We define
\begin{equation*}
\mathcal{I}( \mu|m):=\IX  \Tb(\log \rho)\, d\mu \in [0, \infty],\quad \mu= \rho\,m\in\PX.
\end{equation*}
We have seen at \eqref{eqd-11a} and \eqref{eqd-11}  that $\Tb\ge 0$ (remark that $\ts\ge0$), this shows that $ \mathcal{I}(\mu|m)\ge0.$

\begin{theorem}\label{resd-07}
The Assumptions \ref{ass-05} are supposed to hold.
Then,   $\Lim t \mathcal{I}(\mu_t|m)=0.$
\\
Let us assume in addition that  the initial density $\rho_0:=d\mu_0/dm$ is bounded: $\sup_\XX \rho_0< \infty.$
\\
Then,  $\Lim t H( \mu_t|m)=0$.
\end{theorem}

\begin{remark}
The total variation norm of a signed measure $\eta$ on $\XX$ is defined by $\| \eta\| _{ \mathrm{TV}}:=|\eta|(\XX)$ and   Csisz\'ar-Kullback-Pinsker inequality is
$\ud\|\mu-m\|_{TV}^2\le	H(\mu|m),$ for any $ \mu\in\PX.$ Therefore we have
$\Lim t\|\mu_t-m\|_{ \mathrm{TV}}= 0.$
\end{remark}

\begin{proof}
The main idea of this proof is to consider the forward heat flow $( \mu_t) _{ t\ge0}$ as an entropic interpolation, see Definition \ref{def-02}, and to apply Claim \ref{resd-04}. It will be seen that the general regularity hypotheses of the theorem ensure that this claim turns into a rigorous statement.

Instead of restricting time to $\ii,$ we allow it to be in $[0,T],$ $T>0,$ and let  $T$ tend to infinity. 
We denote $ \rho_t:= d \mu_t/dm.$
Since $\mu$ is the time-marginal flow of $P=\rho_0(X_0)\,R,$ we have $f_0=\rho_0$ and $g_T=1.$ This implies that $g_t=1$ and $f_t= \rho_t.$ Indeed, $g_t(z)=E_R[g_T(X_T)\mid X_t=z]=E_R(1\mid X_t=z)=1$ and $f_t(z)=E_R[f_0(X_0)\mid X_t=z]=E_R[ \rho_0(X_0)\mid X_t=z]=E_R[ dP/dR\mid X_t=z]=dP_t/dR_t(z)=dP_t/dm(z)=: \rho_t(z).$ Therefore,
$
\psi_t=0, \varphi_t=\log\rho_t,
$
for all $0\le t\le T$.
Our assumptions guarantee that for all $t>0,$ $ \varphi_t$ is regular enough  to use our previous results about ``$ \Theta_2$-calculus''.
Denoting $H(t)=H(\mu_t|m)$, with Corollary \ref{res-02}, we see that 
$$-H'(t)=\Ib(t):= \IX \Tb(\log \rho_t)\, d\mu_t  =:\mathcal{I}(\mu_t|m).$$ 
As $\mathcal{I}(\cdot|m)\ge0$,
$H$ is decreasing. Of course, 
$
H(0)-H(T)=\int_0^T \mathcal{I}(\mu_t|m)\,dt.
$
Now, we let $T$ tend to infinity. As $H$ is non-negative and decreasing, it admits a limit $H(\infty):=\Lim T H(T)$ and 
the integral in
\begin{equation}\label{eq-41}
H(0)-H(\infty)=\int_0^\infty \mathcal{I}(\mu_t|m)\,dt
\end{equation}
is convergent, implying
$	
\Lim t \mathcal{I}(\mu_t|m)=0.
$	

\par\medskip\noindent\emph{Diffusion setting}.\quad
By \eqref{eq-22} with $ \varphi_t= \rho_t,$
this limit implies that $\Lim t \IX |\nabla \sqrt{ \rho_t}|^2 \, d\vol=0.$
By Poincaré's inequality, for any open bounded connected domain $U$ with a smooth boundary, there exists a constant $C_U$ such that     $\|\sqrt{\rho_t}-\langle \sqrt{\rho_t}\rangle_U \|_{L^2(U)}\le C_U\|\nabla \sqrt{\rho_t}\|_{L^2(U)}$ for all $t,$ where $\langle \sqrt{\rho_t}\rangle_U :=\int_U \sqrt{\rho_t}\,d\vol/\vol(U)$. It follows that $\Lim t \|\sqrt{\rho_t}-\langle \sqrt{\rho_t}\rangle_U \|_{L^2(U)}=0$  which in turns implies that $\Lim t \rho_t=c$  everywhere on $\XX$ (recall that $\rho$ is continuous on $(0,\infty)\times \XX$ and $\XX$ is connected) for some constant $c\ge 0.$   Therefore, $\Lim t \rho_t=1$  everywhere. As $ \rho_0$ is assumed to be bounded, for all $t$ we have $ \rho_t(X_t)=E_R( \rho_0\mid X_t)\le \sup \rho_0< \infty$. This allows us to apply Lebesgue dominated convergence theorem to assert that $\Lim t H( \mu_t|m)=0.$
\par\medskip\noindent\emph{Random walk setting}.\quad
Since $\Lim t \mathcal{I}(\mu_t|m)=0$, we obtain
\begin{equation}\label{eqd-15}
\Lim t   \sum _{(x,y):x\sim y} \rho_t(x)\ts\big(\rho_t(y)/\rho_t(x)-1\big)\,m(x)\Jb_x(y)=0.
\end{equation} 
We have $$a\ts(b/a-1)\ge \left\{\begin{array}{lcll}
C\frac{(b-a)^2}{a}&\ge&C(b-a)^2,& \textrm{if }b/a\le 2,a\le1\\
\ts(b/a-1)&\ge&C(b/a-1)^2,&\textrm{if }b/a\le2,a\ge1\\
C\log(b/a),&&&\textrm{if }b/a\ge2,a\ge2,
\end{array}\right.$$
for some constant\footnote{$C=2$ is all right, but we do not seek precision.}  $C>0.$ In these three cases, we see that the convergence  to zero of the left-hand side implies the convergence of $b$ to $a.$ The remaining case when $b/a\ge2$ and $a\le2$ is controlled by considering the symmetric term $b\ts(a/b-1)$ which also appears in the  series  in the limit \eqref{eqd-15}, reverting $x$ and $y$; recall that $\Jb_x(y)>0\Leftrightarrow \Jb_y(x)>0.$ As  $b\ts(a/b-1)\ge b\ge 2a\ge 0,$ the convergence of the left-hand term to zero also implies the convergence of $b$ to $a.$ Therefore, the limit \eqref{eqd-15} implies that for all adjacent $x$ and $y\in\XX,$ $\Lim t \rho_t(x)=\Lim t \rho_t(y).$ Since the graph is connected, we have $\Lim t \rho_t=1$ everywhere. Finally, the estimate $\sup _{ t,x} \rho_t(x)\le \sup_\XX \rho_0< \infty$ allows us to conclude with the dominated convergence theorem.
\end{proof}

\begin{theorem}\label{res-10}

The Assumptions \ref{ass-05} are supposed to hold. Then,  under the additional hypotheses that $H( \mu_0|m)< \infty$ and
\begin{equation}\label{eq-35}
\Tb_2\ge \kappa  \Tb
\end{equation}
for some constant $\kappa >0,$ we have 
\begin{eqnarray}
 \mathcal{I}(\mu_t|m)&\le& \mathcal{I}(\mu_0|m) e ^{ -\kappa t},\quad t\ge0 \label{eq-39a}\\
H( \mu_t|m)&\le& H( \mu_0|m) e ^{ -\kappa t},\quad t\ge0 \label{eq-39b}
\end{eqnarray}
and the following functional inequality
\begin{equation}\label{eq-42}
H( \rho m|m)\le \kappa  ^{ -1} \mathcal{I}( \rho m|m),
\end{equation}
holds for any nonnegative  function $ \rho:\XX\to[0, \infty)$ such that $\IX \rho\, dm=1$, which in the diffusion setting (a) is also assumed to be $ \mathcal{C}^2$-regular.
\end{theorem}

\begin{proof}[Proof of Theorem \ref{res-10}]

We already saw during the proof of Theorem \ref{resd-07} that our general assumptions allow us to apply  $ \Theta_2$-calculus: Claim \ref{resd-04} is rigorous with $ \varphi_t=\log \rho_t.$ This gives:
$\Ib(t)=-H'(t)= \IX  \Tb(\log \rho_t)\, d\mu_t  $ and $\Ib'(t)=-H''(t)=  -\IX \Tb_2(\log \rho_t)\, d\mu_t.$
The inequality \eqref{eq-35} implies that $\Ib'(t)\le -\kappa \Ib(t),$ for all $t\ge0.$ Integrating leads to $\Ib(t)\le \Ib(0) e ^{ -\kappa t},$ $t\ge0$, which is \eqref{eq-39a}.

Let us assume for a while, in addition to \eqref{eq-35},  that $\sup_\XX \rho_0< \infty$. 
It is proved at Theorem \ref{resd-07}  that under this restriction, $H(\infty)=0.$
  With \eqref{eq-41}  and \eqref{eq-39a}, we see that
$$
H( \mu_0|m)=H(0)=\int_0 ^{ \infty}\Ib(t)\,dt\le\Ib(0) \int_0 ^{ \infty} e ^{ -\kappa t}\,dt
=\Ib(0)/\kappa.
$$
 A standard approximation argument based on the sequence $ \rho_n=( \rho\wedge n)/\IX ( \rho\wedge n)\,dm$  and Fatou's lemma: $H( \rho m|m)\le\liminf_n H( \rho_nm|m)$, allows to extend this inequality to unbounded $ \rho.$ This proves \eqref{eq-42}. 
\\
Plugging $ \mu_t$ instead of $ \mu_0$ into \eqref{eq-42}, one sees that $H(t)\le \Ib(t)/\kappa=-H'(t)/\kappa.$ Integrating leads to $H(t)\le H(0)e ^{ -\kappa t}$, which is \eqref{eq-39b}.
\end{proof}

\begin{remark}[About time reversal] \label{rem-05}
It appears with \eqref{eq-35} that time reversal is tightly related to this convergence to stationarity. Let us propose an informal interpretation of this phenomenon. The forward entropy production $\If(t)$  vanishes along the forward heat flow $( \mu_t)_{ t\ge0}$ since $\If(t)=\IX \Tf \psi_t\, d \mu_t$ and $ \psi=0.$ Similarly $\IX \Tf_2 \psi_t\, d \mu_t=0.$  No work is needed to drift along the heat flow. In order to evaluate the rate of convergence, one must measure the strength of the drift toward equilibrium. To do so ``one has to face the wind'' and measure the work needed to reach $ \mu_0$ when starting from $m$, reversing  time. 
\end{remark}

\begin{remarks}[About \ $\Tb_2\ge \kappa  \Tb$]\ 
\begin{enumerate}[(a)]
\item
In the diffusion setting, \eqref{eq-35} writes 
\begin{equation*}
\Gb_2\ge \kappa  \Gamma.
\end{equation*}
When $\Lf$ is the reversible forward derivative \eqref{eq-11}, it is the standard Bakry-\'Emery curvature condition $ \mathrm{CD}(\kappa, \infty)$, see \cite{BE85,Ba92}. Further detail is given below at Theorem \ref{res-11}.

\item
In the random walk setting, we see with Lemma \ref{res-07} that \eqref{eq-35} writes 
\begin{align}\label{eq-36}
\Big(\sum _{ y:x\sim y} (e^{\aa_y} -1)&\Jb _x (y)\Big)^2+
	\sum _{ y:x\sim y}[\Jb_y(\XX)-\Jb_x(\XX)] h(\aa_y )\Jb _x (y)\hskip 4cm\notag\\
	&\hskip 3cm+\sum _{ (y,z):x\sim y\sim z} [2e^{\aa_y} h(\cc_z-\aa_y )-h(\cc_z )]\Jb_x(y)\Jb_y(z)\\
&\ge \kappa\sum _{ y:x\sim y} h( \aa_y)\,\Jb _{ x}(y)\notag
\end{align}
for  all $x\in\XX$ and any  numbers $ \aa_y, \cc_z\in\RR,$ $(y,z):x\sim y\sim z,$ where $h(a):= \ts(e^a-1)=ae^a-e^a+1,$ $a\in\RR,$ see \eqref{eq-37}.
\item
An inspection of the proof of Theorem \ref{res-10} shows that only the integrated version of \eqref{eq-35}:
\begin{equation*}
\IX \Tb_2( \rho)\,d \mu\ge \kappa \IX \Tb(\rho)\,d \mu,\quad \mu= \rho m\in\PX
\end{equation*}
is sufficient.
\end{enumerate}
\end{remarks}

In the diffusion setting \eqref{eq-46} on $\XX=\Rn$, we know with \eqref{eq-17b} that \eqref{eq-35} becomes
\begin{equation*}
\| \nabla^2 u\|^2_{\mathrm{HS}}+\Big(\nabla^2V+[\nabla + \nabla^*] b _{ \bot}\Big)(\nabla u)
	\ge \kappa |\nabla u|^2
\end{equation*}
 for any sufficiently regular function $u.$  This $ \Gamma_2$-criterion was obtained by Arnold, Carlen and Ju \cite{ACJ08}; see also the paper \cite{HHS05} by Hwang,  Hwang-Ma  and Sheu for a related result. These results are recovered in the recent paper 
\cite{FJ13} by Fontbona and Jourdain who  implement a stochastic process approach, slightly different from the present article's one but where time reversal plays also a crucial role, see Remark \ref{rem-05}.

\subsection*{Reversible dynamics}

More precisely, we are concerned with the following already encountered $m$-reversible generators $L$.
\begin{enumerate}[(a)]
\item
On a Riemannian manifold $\XX$, see \eqref{eq-11}:
\begin{equation}\label{eq-11b}
L=(-\nabla V\cdot \nabla+\Delta)/2
\end{equation}
with the reversing measure $m=e ^{ -V}\vol$.
\item
On a graph $(\XX,\sim)$, see \eqref{eq-14}:
\begin{equation}\label{eq-14b}
Lu(x)=\sum _{ y:x\sim y}[u(y)-u(x)] s(x,y) \sqrt{m_y/m_x} 
\end{equation}
with $s(x,y)=s(y,x)>0$ for all $x\sim y.$ 
\end{enumerate}

Let us recall our hypotheses.

\begin{assumptions}\label{ass-04}
It is required that the reversing measure $m$ is a probability measure.
\\
 The assumptions (a) and (b) below allow us to apply respectively Theorems \ref{resd-06} and  \ref{res-09}.
\begin{enumerate}[(a)]
\item
The Riemannian manifold $\XX$ is  compact connected   without boundary and $L$ is given at \eqref{eq-11b}. We assume that $V:\XX\to\RR$  is $\mathcal{C}^4$-regular and  without loss of generality that $V$ is normalized by an additive constant such that $m=e ^{ -V}\vol$ is a probability measure. 
\item
The countable graph $(\XX,\sim)$ is assumed to be locally finite and connected.  The generator $L$ is given at \eqref{eq-14b} and the equilibrium probability measure $m$ and the function $s$ satisfy \eqref{eq-21}.
\end{enumerate}
\end{assumptions}

Reversibility allows for a  simplified expression of $ \mathcal{I}(\cdot|m).$ Indeed, as
$ \rho \Theta(\log \rho)=L( \rho\log \rho- \rho)-\log \rho\, L \rho ,$ we obtain
\begin{equation*}
\IX \Theta(\log \rho) \rho\, dm
=\IX L( \rho\log \rho- \rho)\,dm-\IX\log \rho\, L \rho\,dm=\ud\IX \Gamma( \rho,\log \rho)\, dm
\end{equation*}
where the last equality follows from the symmetry   of $L$ in $L^2(m)$ (a consequence of the $m$-reversibility of $R),$ whenever $( \rho,\log \rho)\in\dom \Gamma.$ Therefore,
\begin{equation} \label{eq-40}
 \mathcal{I}(\mu|m)=I( \mu|m):=\ud\IX \Gamma( \rho,\log \rho)\, dm.
\end{equation}
It is the Fisher information of $ \mu$ with respect to $m$.
A direct computation shows that
\begin{enumerate}[(a)]
\item
In the diffusion setting (a), 
\begin{equation*}
I( \mu|m)=\ud\IX |\nabla \log \rho|^2 \, d \mu=\ud\IX \frac{|\nabla \rho|^2}{ \rho}\,dm=2\IX | \nabla\sqrt{ \rho}|^2\,dm;
\end{equation*}
\item
 In the random walk setting (b),
\begin{equation*}
I( \mu|m)=\ud \sum _{(x,y):x\sim y}[ \rho(y)- \rho(x)][\log \rho(y)-\log \rho(x)]\, m(x)J_x(y).
\end{equation*}
\end{enumerate}

Theorem \ref{res-10} becomes

\begin{theorem}\label{res-11}
The Assumptions \ref{ass-04} are supposed to hold. Then,  under the additional hypotheses that $H( \mu_0|m)< \infty$ and
\begin{equation}\label{eq-44}
\Theta_2\ge \kappa  \Theta
\end{equation}
for some constant $\kappa >0,$ we have 
\begin{eqnarray*}
I( \mu_t|m)&\le& I( \mu_0|m) e ^{ -\kappa t},\quad t\ge0\\
H( \mu_t|m)&\le& H( \mu_0|m) e ^{ -\kappa t},\quad t\ge0
\end{eqnarray*}
and the following (modified) logarithmic Sobolev inequality
\begin{equation}\label{eq-45}
H( \mu|m)\le \kappa  ^{ -1} I( \mu|m),
\end{equation}
for any $ \mu\in\PX$ which in the diffusion setting (a) is also restricted to be such that $ d \mu/dm$ is $ \mathcal{C}^2$-regular.
\end{theorem}

In the diffusion setting (a), Theorem \ref{res-11}  is covered by  the well-known result by Bakry and \'Emery \cite{BE85}. Inequality \eqref{eq-45} is the standard logarithmic Sobolev inequality and \eqref{eq-44} is the usual  $ \Gamma_2$-criterion: $ \Gamma_2\ge \kappa \Gamma.$ In the random walk setting (b), inequality \eqref{eq-45} is a modified logarithmic Sobolev inequality which was introduced by Bobkov and Tetali  \cite{BT06} in a general setting and was already employed by Dai Pra, Paganoni and Posta  \cite{DPPP02} in the context of Gibbs measures on $ \mathbb{Z}^d.$ Later Caputo, Dai Pra and Posta
\cite{CP07,CDPP09,DPP13} have derived explicit criteria  for \eqref{eq-45} to be satisfied in specific settings related to Gibbs measures of particle systems. These criteria are  close in spirit to \eqref{eq-44}.

\section{Some open questions about curvature and entropic interpolations}\label{sec-questions}

We have  seen at Section \ref{sec-equilib} that convergence to equilibrium is obtained by considering heat flows. The latter  are the simplest entropic interpolations since $f=1$ or $g=1$. In the present section, general  entropic interpolations are needed to explore curvature properties of Markov generators  and their underlying state space $\XX$ by means of the main result of the article:
\begin{equation*}
H'(t)= \IX \ud (\Tf \psi_t-\Tb \varphi_t )\, d\mu_t,
\quad
H''(t)= \IX\ud(\Tf_2 \psi_t+\Tb_2 \varphi_t)\,d \mu_t.
\end{equation*} 

Convexity properties of the entropy along displacement interpolations are fundamental for  the Lott-Sturm-Villani (LSV) theory of lower bounded curvature of metric measure spaces \cite{St06a,St06b,LV09,Vill09}. 
An alternate  approach  would be to  replace displacement interpolations by entropic interpolations, taking advantage of the analogy between  these two types of interpolations.  Although this program is interesting in itself, it can be further motivated by  remarking the following facts. 
\begin{enumerate}
\item
Entropic interpolations are  more regular than displacement interpolations.
\item
Displacement interpolations are (semiclassical) limits of entropic interpolations, see \cite{Leo12a,Leo12c}.
\item
Entropic interpolations work equally well in continuous and discrete settings while LSV theory fails in the discrete setting (see below).
\end{enumerate}

As regards (1), note that the dynamics of an entropic interpolation shares the regularity of the solutions $f$ and $g$ of the backward and forward ``heat equations'' \eqref{eq-27}. Their logarithms, the Schrödinger potentials $ \varphi$ and $ \psi$, solve \emph{second order} Hamilton-Jacobi-Bellman equations.  This is in contrast with a displacement interpolation whose dynamics is driven by the Kantorovich potentials (analogues of $ \varphi$ and $ \psi$) which are  solutions of \emph{first order} Hamilton-Jacobi equations. Entropic interpolations are linked to regularizing dissipative PDEs, while displacement interpolations are linked to transport PDEs. More detail about these relations is given in \cite{Leo12e}. 

In practice, once the regularity of the solutions of the dissipative linear PDEs \eqref{eq-27} is ensured, the rules of calculus that are displayed at the beginning of Section \ref{sec-second} are efficient. Again, this is in contrast with Otto calculus which only yields heuristics about displacement interpolations.

As regards (3), recall that we have seen at Theorems \ref{res-10} and \ref{res-11} 
unified proofs of entropy-entropy production inequalities, available   in the diffusion and random walk settings.

In order to provide a better understanding of the analogy between displacement and entropic interpolations, we start describing two thought experiments. Then, keeping the LSV strategy as a guideline,  we raise a few open questions about the use of entropic interpolations in connection with curvature problems on Riemannian manifolds and graphs.

\subsection*{Thought experiments}

Let us describe  two thought experiments.
\begin{itemize}
\item \emph{Schrödinger's hot gas experiment}.\ 
The dynamical Schrödinger problem (see \eqref{Sdyn} at Section \ref{sec-inter}) is a formalization of Schrödinger's  thought experiment which was introduced in  \cite{Sch31}. Ask a perfect gas of  particles living  in a Riemannian manifold $\XX$ which are in contact with a \emph{thermal bath}   to start from a configuration profile $ \mu_0\in\PX$ at time $t=0$ and to end up at the unlikely profile $ \mu_1\in\PX$ at time $t=1.$ For instance the gas may be constituted of undistinguishable mesoscopic particles (e.g.\ grains of pollen) with a random motion communicated by  numerous shocks with the microscopic particles  of the thermal bath (e.g.\ molecules of hot water). Large deviation considerations lead to the following conclusion. In the (thermodynamical) limit of infinitely many mesoscopic particles, the most likely trajectory of the whole mesoscopic particle system  from $ \mu_0$ to $ \mu_1$ is the $R$-entropic interpolation $[ \mu_0,\mu_1]^R$ where $R$ describes the random motion of the non-interacting  (the gas is assumed to be perfect) mesoscopic particles. Typically, $R$ describes a Brownian motion as in Schrödinger's original papers \cite{Sch31,Sch32}. For more detail see \cite[\S 6]{Leo12e} for instance.
\item
\emph{Cold gas experiment}.\ 
The same thought experiment is also described in \cite[pp.\,445-446]{Vill09} in a slightly different setting where it is called the \emph{lazy gas experiment}. The only difference with Schrödinger's thought experiment is that no thermal bath enters the game.  The perfect gas is cold so that the paths of the particles are Riemannian  geodesics in order to minimize the average kinetic action of the whole system:  entropy minimization is replaced by quadratic optimal transport. 
\end{itemize}

\subsubsection*{Connection between entropic and displacement interpolations} It is shown in \cite{Leo12a,Leo12c} that displacement interpolations are  limits of entropic interpolations when the frequency of the random elementary motions encrypted by $R$ tends down to zero. With the above thought experiment in mind, this corresponds to a thermal bath whose density tends down to zero. Replacing the hot water by an increasingly rarefied gas, at the zero-temperature limit,  the entropic interpolation $[ \mu_0,\mu_1]^R$ is replaced by McCann's displacement interpolation $[ \mu_0,\mu_1] ^{ \textrm{disp}}$.

\subsubsection*{Cold  gas experiment }
As the gas is (absolutely) cold, the particle sample paths are regular and deterministic.
The trajectory of the mass distribution of the cold gas is the displacement interpolation $[ \mu_0,\mu_1]^{ \textrm{disp}}$ which is the time marginal flow of a probability measure $P$ on $\OO$ concentrated on geodesics:
\begin{equation}\label{eq-04}
P=\IXX \delta _{ \gamma ^{ xy}}\, \pi(dxdy)\in\PO
\end{equation}
where $ \pi\in\PXX$ is an optimal transport plan between $ \mu_0$ and $ \mu_1$, see \cite{Leo12e}. Therefore,
\begin{equation}\label{eq-04b}
\mu_t=P_t=\IXX \delta _{ \gamma ^{ xy}_t}\, \pi(dxdy)\in\PX,\quad 0\le t\le1,
\end{equation}
where for simplicity it is assumed that there is a unique minimizing  geodesic $ \gamma ^{ xy}$ between $x$ and $y$ (otherwise, one is allowed to replace $ \delta _{ \gamma ^{ xy}}$ with any probability concentrated on the set of all minimizing geodesics from $x$ to $y$).

In a  positively curved manifold, several geodesics starting from the same point have a tendency to approach each other. Therefore, it is necessary that the gas initially spreads out to thwart its tendency to concentrate. Otherwise, it would not be possible to reach a  largely spread target distribution.
The left side of Figure 1 below (taken from \cite{Vill09}) depicts this phenomenon. This is also the case of the right side as one can be convinced of by reversing time. The central part is obtained by interpolating between the initial and final segments of the geodesics.  
\begin{center}
\scalebox{0.8} 
{
\begin{pspicture}(0,-3.7578413)(10.148961,3.7278414)
\definecolor{color109b}{rgb}{0.0,0.4,1.0}
\definecolor{color594b}{rgb}{0.0,0.6,1.0}
\definecolor{color20b}{rgb}{0.0,0.2,0.8}
\psbezier[linewidth=0.04,fillstyle=solid,fillcolor=color109b](9.14,2.1895804)(8.151038,2.3377533)(7.4420795,0.18333259)(7.8,-0.7504195)(8.157921,-1.6841716)(9.22,-2.1904194)(9.5,-1.5704195)(9.78,-0.9504195)(8.527898,-0.7441046)(8.34,0.2095805)(8.152102,1.1632656)(10.128962,2.0414078)(9.14,2.1895804)
\psbezier[linewidth=0.04,fillstyle=solid,fillcolor=color594b](4.2837486,3.131174)(3.5237472,3.0527675)(3.3574972,1.131808)(3.309997,0.661369)(3.262497,0.19092995)(3.12,-3.1110413)(4.64,-3.1609483)(6.16,-3.2108552)(5.376203,-1.3818146)(4.56872,-0.5470155)(3.7612367,0.28778362)(4.164978,1.2617159)(4.4974713,1.7586201)(4.829964,2.2555244)(5.0437503,3.2095804)(4.2837486,3.131174)
\psellipse[linewidth=0.04,dimen=outer,fillstyle=solid,fillcolor=color20b](0.56,-0.1304195)(0.56,1.62)
\psbezier[linewidth=0.02,arrowsize=0.093cm 2.0,arrowlength=1.4,arrowinset=0.4,dotsize=0.07055555cm 2.0]{*->}(0.58,1.1695805)(2.2558823,3.2095804)(6.9237494,3.7178414)(8.72,1.7295805)
\psbezier[linewidth=0.02,arrowsize=0.093cm 2.0,arrowlength=1.4,arrowinset=0.4,dotsize=0.07055555cm 2.0]{*->}(0.52,0.2695805)(3.6986363,1.0295805)(5.7175,0.9155805)(8.08,0.3455805)
\psbezier[linewidth=0.02,arrowsize=0.093cm 2.0,arrowlength=1.4,arrowinset=0.4,dotsize=0.07055555cm 2.0]{*->}(0.52,-0.2304195)(1.6704421,-1.5511296)(6.536667,-1.7304195)(8.12,-0.55491656)
\psbezier[linewidth=0.02,arrowsize=0.093cm 2.0,arrowlength=1.4,arrowinset=0.4,dotsize=0.07055555cm 2.0]{*->}(0.56,-1.4839579)(3.2461472,-3.1904194)(7.42,-3.1504195)(9.06,-1.3104194)
\psdots[dotsize=0.12](4.42,2.9295805)
\psdots[dotsize=0.12](3.7,0.7695805)
\psdots[dotsize=0.12](4.28,-1.3704195)
\psdots[dotsize=0.12](4.66,-2.7304194)
\psdots[dotsize=0.12](9.06,-1.3304195)
\psdots[dotsize=0.12](8.12,-0.53041947)
\psdots[dotsize=0.12](8.1,0.3895805)
\psdots[dotsize=0.12](8.7,1.7095805)
\usefont{T1}{ptm}{m}{n}
\rput(0.46145508,-3.5254195){$t=0$}
\usefont{T1}{ptm}{m}{n}
\rput(4.631455,-3.5854194){$0<t<1$}
\usefont{T1}{ptm}{m}{n}
\rput(9.061455,-3.5654194){$t=1$}
\end{pspicture} 
}
\end{center}
\medskip
\begin{center}
Figure 1.\quad
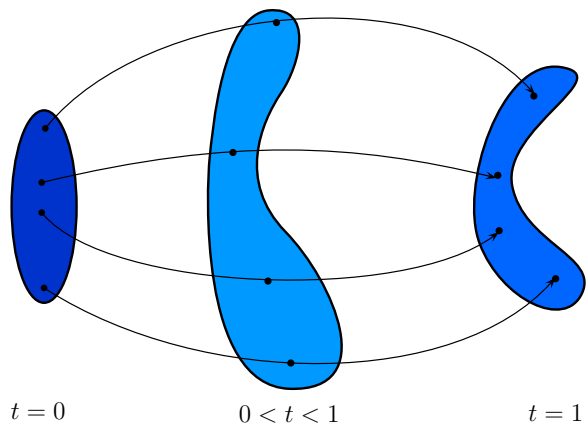
\emph{The cold gas experiment.  Positive curvature}
\end{center}

On the other hand, in a negatively curved manifold, several geodesics starting from the same point have a tendency to depart form each other. Therefore, it is necessary that the gas initially concentrates to thwart its tendency to spread out. Otherwise, it would not be possible to reach a  condensed target distribution.

\begin{center}
\scalebox{0.8} 
{
\begin{pspicture}(0,-2.5377972)(10.148961,2.5177972)
\definecolor{color385b}{rgb}{0.0,0.4,0.8}
\definecolor{color388b}{rgb}{0.0,0.8,1.0}
\definecolor{color391b}{rgb}{0.0,0.6,1.0}
\psbezier[linewidth=0.04,fillstyle=solid,fillcolor=color385b](4.76,0.37983638)(4.58,0.66962457)(4.1620793,-0.069495946)(4.2707467,-0.27570528)(4.379414,-0.48191458)(4.7213593,-0.5903754)(4.800692,-0.38691857)(4.8800244,-0.18346176)(4.8395977,-0.2576977)(4.64,-0.0065479106)(4.440402,0.24460188)(4.94,0.090048164)(4.76,0.37983638)
\psbezier[linewidth=0.04,fillstyle=solid,fillcolor=color388b](9.14,2.3496246)(8.151038,2.4977973)(7.4420795,0.34337667)(7.8,-0.5903754)(8.157921,-1.5241275)(9.22,-2.0303755)(9.5,-1.4103754)(9.78,-0.7903754)(8.527898,-0.5840605)(8.34,0.36962458)(8.152102,1.3233097)(10.128962,2.2014518)(9.14,2.3496246)
\psellipse[linewidth=0.04,dimen=outer,fillstyle=solid,fillcolor=color391b](0.56,0.029624587)(0.56,1.62)
\psbezier[linewidth=0.02,arrowsize=0.093cm 2.0,arrowlength=1.4,arrowinset=0.4,dotsize=0.07055555cm 2.0]{*->}(0.6,1.0496246)(2.76,0.14962459)(5.4,-0.37037542)(8.66,1.8706614)
\psbezier[linewidth=0.02,arrowsize=0.093cm 2.0,arrowlength=1.4,arrowinset=0.4,dotsize=0.07055555cm 2.0]{*->}(0.52,-0.07037541)(0.7,0.08962459)(7.4,-0.090375416)(8.12,-0.39487246)
\psbezier[linewidth=0.02,arrowsize=0.093cm 2.0,arrowlength=1.4,arrowinset=0.4,dotsize=0.07055555cm 2.0]{*->}(0.56,-1.3239138)(3.0,0.7696246)(8.16,-1.0303754)(9.06,-1.1503754)
\psdots[dotsize=0.12](9.06,-1.1703755)
\psdots[dotsize=0.12](8.12,-0.37037542)
\psdots[dotsize=0.12](8.1,0.54962456)
\psdots[dotsize=0.12](8.7,1.8696246)
\usefont{T1}{ptm}{m}{n}
\rput(0.4814551,-2.3653755){$t=0$}
\usefont{T1}{ptm}{m}{n}
\rput(4.2914553,-2.3653755){$0<t<1$}
\usefont{T1}{ptm}{m}{n}
\rput(8.961455,-2.3453753){$t=1$}
\psbezier[linewidth=0.02,arrowsize=0.093cm 2.0,arrowlength=1.4,arrowinset=0.4,dotsize=0.07055555cm 2.0]{*->}(0.52,0.22962458)(2.02,-0.07037541)(5.5,-0.090375416)(8.02,0.5296246)
\psdots[dotsize=0.12](4.62,0.2896246)
\psdots[dotsize=0.12](4.48,0.06962459)
\psdots[dotsize=0.12](4.54,-0.07037541)
\psdots[dotsize=0.12](4.52,-0.3303754)
\end{pspicture} 
}
\end{center}
\medskip
\begin{center}
Figure 2.\quad
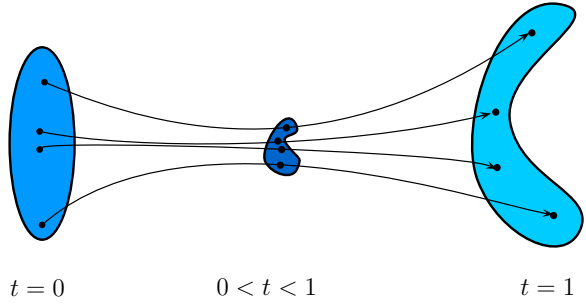
\emph{The cold gas experiment.  Negative curvature}
\end{center}

\subsubsection*{Schrödinger's hot
 gas experiment}

As the gas is hot, the particle sample paths are irregular and random.
The trajectory of the mass distribution of the hot gas is the entropic interpolation $[ \mu_0,\mu_1]^R.$
An  $R$-entropic interpolation is the time marginal flow of a probability measure $P$ on $\OO$ which is a mixture of bridges of $R$: 
\begin{equation}\label{eq-05}
P=\IXX R ^{ xy}\, \pi(dxdy)\in\PO
\end{equation}
where $ \pi\in\PXX$ is the unique minimizer of $ \eta\mapsto H( \eta| R _{ 01})$ among all $\eta\in\PXX$ with prescribed marginals $ \mu_0$ and $ \mu_1$, see \cite{Leo12e}. Therefore, 
\begin{equation} \label{eq-05b}
 \mu_t=P_t=\IXX R ^{ xy}_t\, \pi(dxdy)\in\PX,\quad 0\le t\le1.
\end{equation}
Comparing \eqref{eq-04} and \eqref{eq-05}, one sees that the deterministic evolution $ \delta _{ \gamma ^{ xy}}$ is replaced by the random evolution $R ^{ xy}.$
\begin{center}
\scalebox{0.8} 
{
\begin{pspicture}(0,-1.8938416)(10.082911,1.8588417)
\definecolor{color612f}{rgb}{0.4,0.4,0.4}
\psdots[dotsize=0.12](9.201015,0.5735802)
\psbezier[linewidth=0.01,fillstyle=gradient,gradlines=2000,gradbegin=white,gradend=color612f,gradmidpoint=0.49,gradangle=13.5,linestyle=dotted,dotsep=0.16cm](0.8810156,-0.8464198)(0.8810156,-1.2664198)(2.8241029,-1.1864197)(3.463266,-1.164642)(4.102429,-1.1428642)(6.187588,-0.7942931)(7.0210156,-0.4895309)(7.854443,-0.18476868)(9.261016,0.23358022)(9.241015,0.5135802)(9.221016,0.79358023)(8.941756,0.87269133)(7.81583,1.0131358)(6.6899047,1.1535802)(2.8610156,0.5535802)(2.2010157,0.23358022)(1.5410156,-0.08641978)(0.8810156,-0.4264198)(0.8810156,-0.8464198)
\psbezier[linewidth=0.02,arrowsize=0.093cm 2.0,arrowlength=1.4,arrowinset=0.4,dotsize=0.07055555cm 2.0]{*->}(0.90101564,-0.8064198)(4.237129,-0.18641979)(4.1210155,0.31358021)(9.1610155,0.5535802)
\usefont{T1}{ptm}{m}{n}
\rput(0.2324707,-0.6014198){$x$}
\usefont{T1}{ptm}{m}{n}
\rput(9.77247,0.6785802){$y$}
\usefont{T1}{ptm}{m}{n}
\rput(3.2024708,1.5985802){$\gamma^{xy}$}
\usefont{T1}{ptm}{m}{n}
\rput(6.642471,-0.9814198){$R^{xy}$}
\psbezier[linewidth=0.01,arrowsize=0.093cm 2.0,arrowlength=1.4,arrowinset=0.4]{->}(3.6810157,1.5324552)(4.021568,1.571386)(3.9464731,1.537582)(4.4440575,1.6957119)(4.9416413,1.8538417)(5.4010158,1.0956216)(5.2610154,0.27358022)
\psdots[dotsize=0.12](9.181016,0.5535802)
\usefont{T1}{ptm}{m}{n}
\rput(0.9224707,-1.7214198){$t=0$}
\usefont{T1}{ptm}{m}{n}
\rput(9.22247,-1.7214198){$t=1$}
\end{pspicture} 
}
\end{center}
\medskip
\begin{center}
Figure 3.\quad
\emph{Hot gas experiment with $ \mu_0= \delta_x$ and $ \mu_1= \delta_y$}
\end{center}
The grey leaf in this figure is a symbol for the (say) 95\%-support of the bridge $R ^{ xy}.$ It spreads around the geodesic $ \gamma ^{ xy}$ and in general the support of  $R ^{ xy}_t$ for each intermediate time $0<t<1$ is the whole space $\XX$. At a lower temperature, $R ^{ xy}$ is closer to $ \delta _{ \gamma ^{ xy}}:$
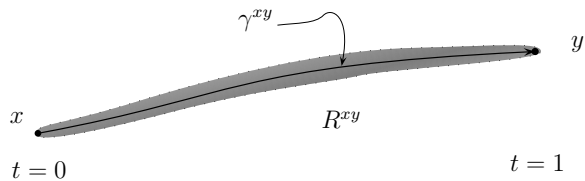
\begin{center}
\scalebox{0.8} 
{
\begin{pspicture}(0,-1.4638417)(9.86291,1.4288416)
\definecolor{color467f}{rgb}{0.4,0.4,0.4}
\psbezier[linewidth=0.01,fillstyle=gradient,gradlines=2000,gradbegin=white,gradend=color467f,gradmidpoint=0.49,gradangle=13.500004,linestyle=dotted,dotsep=0.16cm](0.5410156,-0.6764198)(0.5200018,-0.83641976)(1.2810156,-0.71641976)(2.8210156,-0.31641978)(4.361016,0.08358022)(5.4010158,0.16358022)(5.9210157,0.2835802)(6.4410157,0.40358022)(8.941015,0.5235802)(8.941015,0.6835802)(8.941015,0.84358025)(8.321015,0.78358024)(6.801016,0.7435802)(5.2810154,0.7035802)(3.5010157,0.24358022)(2.2010157,-0.15641978)(0.90101564,-0.5564198)(0.5620294,-0.51641977)(0.5410156,-0.6764198)
\psbezier[linewidth=0.02,arrowsize=0.093cm 2.0,arrowlength=1.4,arrowinset=0.4,dotsize=0.07055555cm 2.0]{*->}(0.58101565,-0.6764198)(3.9171293,-0.056419782)(3.8010156,0.4435802)(8.841016,0.6835802)
\usefont{T1}{ptm}{m}{n}
\rput(0.2324707,-0.41141978){$x$}
\usefont{T1}{ptm}{m}{n}
\rput(9.552471,0.8085802){$y$}
\usefont{T1}{ptm}{m}{n}
\rput(4.202471,1.2085803){$\gamma^{xy}$}
\usefont{T1}{ptm}{m}{n}
\rput(5.622471,-0.45141977){$R^{xy}$}
\psbezier[linewidth=0.01,arrowsize=0.093cm 2.0,arrowlength=1.4,arrowinset=0.4]{->}(4.5610156,1.1224554)(4.901568,1.1613861)(5.186473,1.147582)(5.324057,1.2857119)(5.461642,1.4238417)(5.7810154,1.2856216)(5.62758,0.4435802)
\psdots[dotsize=0.12](8.841016,0.6835802)
\usefont{T1}{ptm}{m}{n}
\rput(0.6024707,-1.2914197){$t=0$}
\usefont{T1}{ptm}{m}{n}
\rput(8.882471,-1.1714197){$t=1$}
\end{pspicture} 
}
\end{center}
\medskip
\begin{center}
Figure 4.\quad
\emph{Hot gas experiment with $ \mu_0= \delta_x$ and $ \mu_1= \delta_y$.\\ At a lower temperature}
\end{center}
And eventually, at zero temperature $R ^{ xy}$ must be replaced by its determinestic limit $ \delta _{ \gamma ^{ xy}}.$
The superposition of Figures 2 and 3 describes the hot
 gas experiment in a negative curvature manifold:
\begin{center}
\scalebox{0.8} 
{
\begin{pspicture}(0,-2.6237109)(10.2089615,2.588711)
\definecolor{color435b}{rgb}{0.8,0.8,0.8}
\definecolor{color438b}{rgb}{0.0,0.4,0.8}
\definecolor{color441b}{rgb}{0.0,0.8,1.0}
\definecolor{color444b}{rgb}{0.0,0.6,1.0}
\psbezier[linewidth=0.01,fillstyle=solid,fillcolor=color435b](0.6438632,1.623711)(0.22615153,1.523711)(0.0,0.56022984)(0.08470532,-0.43628907)(0.16941065,-1.432808)(0.2976459,-1.6772738)(0.6841046,-1.736289)(1.0705633,-1.7953043)(1.1382631,-0.66546893)(2.1168153,-0.41628906)(3.095368,-0.16710919)(3.6944664,-1.1572857)(4.656953,-1.2962891)(5.6194396,-1.4352924)(6.0315127,-0.1289318)(6.7297053,-0.63628906)(7.4278984,-1.1436464)(8.76,-1.876289)(8.96,-1.876289)(9.16,-1.876289)(9.486873,-1.9020104)(9.58,-1.4162891)(9.673127,-0.9305677)(8.590455,-0.43492073)(8.56,0.38371095)(8.529545,1.2023426)(9.44,1.623711)(9.48,2.103711)(9.52,2.583711)(8.12,2.3837109)(6.8922744,0.84371096)(5.6645484,-0.69628906)(5.6426334,1.0640366)(4.636632,0.92371094)(3.6306303,0.7833853)(3.149923,0.34356624)(2.0964944,0.5037109)(1.0430654,0.6638557)(1.0615748,1.7237109)(0.6438632,1.623711)
\psbezier[linewidth=0.04,linestyle=dashed,dash=0.16cm 0.16cm,fillcolor=color438b](4.82,0.31392273)(4.64,0.60371095)(4.2220798,-0.1354096)(4.3307467,-0.34161893)(4.439414,-0.54782826)(4.781359,-0.65628904)(4.860692,-0.45283222)(4.9400244,-0.2493754)(4.8995976,-0.32361135)(4.7,-0.07246156)(4.500402,0.17868823)(5.0,0.024134517)(4.82,0.31392273)
\psbezier[linewidth=0.04,fillstyle=solid,fillcolor=color441b](9.2,2.283711)(8.211039,2.4318836)(7.5020795,0.27746302)(7.86,-0.65628904)(8.21792,-1.5900412)(9.28,-2.0962892)(9.56,-1.476289)(9.84,-0.8562891)(8.587897,-0.64997417)(8.4,0.30371094)(8.212102,1.257396)(10.188961,2.1355383)(9.2,2.283711)
\psellipse[linewidth=0.04,dimen=outer,fillstyle=solid,fillcolor=color444b](0.62,-0.036289062)(0.56,1.62)
\usefont{T1}{ptm}{m}{n}
\rput(0.5414551,-2.4512892){$t=0$}
\usefont{T1}{ptm}{m}{n}
\rput(9.081455,-2.411289){$t=1$}
\end{pspicture} 
}
\end{center}
\medskip
\begin{center}
Figure 5.\quad
\emph{The hot gas experiment in a negative curvature manifold}
\end{center}
Figures 3, 4 and 5 are only suggestive pictures that describe  most visited space areas.

\subsection*{Open questions}

As far as one is interested in the curvature of the metric measure space $(\XX,d,m)$, rather than in the curvature of some Markov generator $\Lf$ such as in Section  \ref{sec-equilib}, it seems natural to restrict our attention to a reversible reference path measure $R.$ For instance when looking at a Riemannian manifold $\XX$ considered as a metric measure space $(\XX,d,m)$ with the Riemannian metric $d$ and the weighted measure $m=e ^{ -V}\vol,$ the efficient choice as regards  LSV theory  is to take $R$ as the reversible Markov measure with generator $L=( \Delta-\nabla V\cdot\nabla)/2$ (or any positive multiple) and initial measure $m$.

A representative result in the case where the reference measure $m$ is the volume measure ($V=0$) is that relative entropy is convex along any displacement interpolation in a Riemannian manifold with a nonnegative Ricci curvature. Based on  McCann's seminal work  \cite{McC94,McC97} and Otto's heuristic calculus, Otto and Villani have conjectured this result in  \cite{OV00}. This was  proved by Cordero-Erausquin, McCann and Schmuckenschläger in \cite{CMS01}. The converse is also true, as was shown later by Sturm and von Renesse \cite{SvR05}.

In the same spirit, an immediate consequence of Theorem \ref{resd-06} is the following

\begin{corollary} \label{res-13}
Let $R^o$ be the reversible Brownian motion on a  connected Riemannian manifold $\XX$ without boundary. Then, along any entropic interpolation $[ \mu_0, \mu_1] ^{ R^o}$ associated with $R^o$ and $f_0,g_1\in L^2(\vol)$ such that \eqref{eq-13} holds, the entropy $t\in[0,1]\mapsto H( \mu_t|\vol)\in\RR$ is a convex function.
\end{corollary}
Remark that in Theorem \ref{resd-06}, $\XX$ is assumed to be compact and $f_0,g_1$ must be $ \mathcal{C}^2$. This was assumed to ensure the regularity of the interpolation (Proposition \ref{res-03}). But in the present setting this regularity follows from the regularity of the heat kernel. 


\begin{questions} \ 
\begin{enumerate}[(a)]
\item
Is the converse of Corollary \ref{res-13} true?
\item
Is there a way to define a notion of $\kappa$-convexity on $\PX$ along entropic interpolations: $H( \mu_t)\le (1-t)H( \mu_0)+tH( \mu_1)-\kappa C( \mu_0, \mu_1) t(1-t)/2 $, $0\le t\le 1$? What term $C( \mu_0, \mu_1)$ should replace the quadratic transport cost $W^2_2( \mu_0, \mu_1)$ of the LSV and AGS (Ambrosio, Gigli and Savaré \cite{AGS05,AGS12}) theories?
\item
Is there a notion of ``entropic gradient flow'' related to entropic interpolations that would be similar to a  gradient flow  and would allow  interpreting heat flows as ``entropic gradient flows'' of the entropy. This question might be related to the previous one: the entropic cost $C( \mu_0, \mu_1)$ could play the role of a squared distance on $\PX.$
\end{enumerate}
\end{questions}

The LSV theory requires that the metric space $(\XX,d)$ is geodesic. Consequently,  discrete metric graphs are ruled out. Alternate approaches are necessary to develop a theory of lower bounded curvature on discrete metric graphs.

\begin{enumerate}
\item
 Bonciocat and Sturm \cite{BSt09}  have obtained precise   results for a large class of planar graphs by introducing  the notion of $h$-approximate $t$-midpoint interpolations:
$d(x_0,x_t)\le td(x_0,x_1)+h,$ $d(x_t,x_1)\le (1-t)d(x_0,x_1)+h$ where typically $h$ is of order 1.

\item
Erbar and Maas \cite{Maas11,EM11} and independently Mielke \cite{Mie11,Mie13} have shown that  the evolution of a reversible random walk on a graph is the gradient flow of an entropy for some distance on $\PX$ and derived curvature bounds for discrete graphs by following closely the gradient flow strategy developed by Ambrosio, Gigli and Savaré \cite{AGS05} in the setting of the LSV  theory.
\item
Recently, Gozlan, Roberto, Samson and Tetali \cite{GRST12} obtained curvature bounds by studying the convexity of entropy along binomial interpolations. On a 
 geodesic $(x:=z_0, z_1,\dots, z _{ d(x,y)}=:y )$ for the standard graph distance $d$, the binomial interpolation is defined by
$ \mu_t(z_k)= \mathcal{B}(d(x,y),t)(k),$ $0\le k\le d(x,y),$ $0\le t\le1,$ where $ \mathcal{B}(n,p)(k)=\begin{pmatrix}n\\k\end{pmatrix} p^k (1-p) ^{ n-k}$ is the usual binomial weight.
\end{enumerate}

We know by Lemma \ref{res-07} that
\begin{align} \label{eq-47}
\Theta_2 u(x)=\Big(Bu(x)\Big)^2&+\sum _{ y:x\sim y} [J_y(\XX)-J_x(\XX)] h\big(Du(x,y)\big)\,J_x(y)\notag\\&+\sum _{ (y,z):x\sim y\sim z} \Big[2e ^{ Du(x,y)} h\big(Du(y,z) \big)-h\big(Du(x,z) \big)\Big]\,J_x(y)J_y(z)
\end{align}
where $Bu(x)=\sum _{ y:x\sim y}(e ^{ Du(x,y)}-1)\,J_x(y)$ and  $h(a):=\ts(e^a-1)=ae^a-e^a+1,$ $ a\in\RR.$
\\
In the diffusion setting where $L=( \Delta-\nabla V\cdot\nabla)/2$, we have seen at Section \ref{sec-diff} that $ \Theta_2= \Gamma_2/2$ where $ \Gamma_2$ is given by the Bochner formula
\begin{equation} \label{eq-48}
\Gamma_2(u)=\|\nabla^2 u\|^2_{\mathrm{HS}}+\nabla^2V(\nabla u)+\Ric(\nabla u).
\end{equation}

\begin{questions}\ 
\begin{enumerate}
\item
In view of \eqref{eq-47} and \eqref{eq-48}, is \eqref{eq-47} a Bochner formula ? Where is the curvature ?
\item
Are the following definitions relevant?
\begin{enumerate}
\item
The $m$-reversible random walk generator $L$ has curvature   bounded below by $\kappa\in \mathbb{R}$ if
$$\sum_x \Theta_2(u)(x)m(x)\ge  \kappa\sum_x \Theta(u)(x)m(x),\ \forall u.$$
\item
The curvature at $x\in\XX$ of the Markov generator $L$ is 
$$\mathrm{curv}_L(x):=\inf_u \frac{\Theta_2(u)}{\Theta(u)}(x).$$
\end{enumerate}
\end{enumerate}
\end{questions}
A definition of curvature should be justified by its usefulness in terms of rate of convergence to equilibrium, concentration of measure or isoperimetric behavior.

Let us indicate that with the discrete Laplacian $Lu(x)=[u(x+1)-2u(x)+u(x-1)]/2$ on $ \mathbb{Z}$, a rather tedious computation leads to the desired flatness result:
$\mathrm{curv}_L(x) =0,$  for all $x\in\mathbb{Z}.$

\appendix

\section{Basic definitions}\label{sec-def}

This appendix section is part of the articles \cite{Leo12b,LRZ12}. We give it here for the reader's convenience.

\subsection*{Markov measures}

We slightly  extend   the notion of Markov property to unbounded positive measures on $\OO.$ 

\begin{definitions}\label{def-markov}
Any positive  measure on $\OO$ is called a \emph{path measure}.
Let $Q$ be a path measure.
\begin{enumerate}[(a)]
\item
It is said to be a \emph{conditionable path measure} if for any $0\le t\le 1,$ $Q_t$ is a $\sigma$-finite measure on $\XX.$
\item
It is said to be a Markov measure if it is conditionable and for all $0\le t\le 1$ and  $B\in \sigma(X _{[t,1]}),$ we have: $Q(B\mid X _{[0,t]})=Q(B\mid X_t).$
\end{enumerate}
\end{definitions}

When $Q$ has an infinite mass, the notion of conditional expectation must be handled with care. We refer to the following definition.

\begin{definition}[Conditional expectation]\label{def-05}
Let $\phi:\OO\to Y$ be a measurable map. Suppose that $\phi\pf Q$ is a $\sigma$-finite measure on $Y.$ Then, for any $f\in L^p(Q)$ with $p=1, 2$ or $\infty,$ the conditional expectation of $f$ knowing $\phi$ is the unique (up to $Q\ae$ equality) function $E_Q(f\mid \phi):=\theta_f(\phi)\in L^p(Q)$ such that for all  measurable function $h$ on $Y$ in $L ^{p_*}(m)$ where $1/p+1/p_*=1$, we have
$ 
\IO h(\phi)f\,dQ=\IO h(\phi)\theta_f(\phi)\,dQ.
$ 
\end{definition}

It is essential in this definition that the measure $\phi\pf Q$ is $\sigma$-finite on $Y$ to be allowed to invoke Radon-Nikodym theorem when proving the existence of $\theta_f.$ Inspecting the above definition of a  Markov measure, one sees that it is necessary that $Q _{[0,t]}$ and $Q_t$ are $\sigma$-finite for all $t\in\ii$ for the corresponding conditional expectations to be defined. But this is warranted by the requirement that $Q$ is conditionable, as shown by the following result.

\begin{lemma}
Let $Q$ be a path measure and $\mathcal{T}\subset\ii.$ For $Q_\mathcal{T}$ to be $\sigma$-finite, it is enough that $Q_{t_o}$ is $\sigma$-finite for some $t_o\in \mathcal{T}.$
\end{lemma}

\begin{proof}
Let $t_o\in \mathcal{T}$ be such that $Q _{t_o}$ is $\sigma$-finite with $\seq{\XX}n$ an increasing sequence of measurable subsets of $\XX$ such that $Q_{t_o}(\XX_n)<\infty$ and $\cup _{n\ge1}\XX_n=\XX.$ Then, $Q_\mathcal{T}$ is also $\sigma$-finite since $Q _{\mathcal{T}}(X _{t_o}\in\XX_n)=Q _{t_o}(\XX_n)$ for all $n\ge1$ and $\cup _{n\ge1}[X _{\mathcal{T}}(\OO)\cap \left\{X _{t_o}\in\XX_n\right\}]=X_\mathcal{T}(\OO).$
\end{proof}

Consequently, for any Markov measure $Q$ and any $\mathcal{T}\subset\ii$, the conditional expectation $E_Q(\cdot\mid X _{\mathcal{T}})$ is well-defined and since $\OO$ is a Polish space, there exists a regular conditional \emph{probability} kernel $Q(\cdot\mid X_\mathcal{T}):\OO\to\PO$ such that for all $f\in L^1(Q),$ $E_Q(f\mid X_\mathcal{T})=\IO f\,dQ(\cdot\mid X_\mathcal{T}),$ see \cite[Thm.\,10.2.2]{Dud02}.

\subsection*{Stationary path measures}

Let us make precise a couple of other known notions.

\begin{definitions}\label{def-01}
Let $Q$ be a path measure.
\begin{enumerate}[(a)]
\item
It is said to be \emph{stationary} if for all  $0\le t\le1,$ $Q_t=m$, for some $m\in \mathrm{M}_(\XX).$ One says that $Q$ is $m$-stationary.
\item
It is said to be \emph{reversible} if for any subinterval $[a,b]\subset \ii,$ we have $(\mathrm{rev}^{[a,b]})\pf Q _{[a,b]}=Q_{[a,b]}$ where $\mathrm{rev}^{[a,b]}$ is the time reversal on $[a,b]$ which is defined by $\mathrm{rev}^{[a,b]}[\eta](t)=\eta ({[a+b-t]^+})$ for any $\eta\in D([a,b],\XX),$ $t\in[a,b].$ 
\end{enumerate}
\end{definitions}

A reversible path measure is stationary. When $Q$ is reversible, one sometimes says that $m=Q_0=Q_1$ is a reversing measure for the forward kernel $(Q(\cdot\mid X_0=x);x\in\XX)$, or the backward kernel $(Q(\cdot\mid X_1=y);y\in\XX)$ or shortly for $Q$. One also says that $Q$ is $m$-reversible to emphasize the role of the reversing measure.

\subsection*{Relative entropy}

\newcommand{\PY}{ \mathrm{P}(Y)}

This subsection  is a short part of \cite[\S\,2]{Leo12b} which we refer to for more detail.
Let $r$ be some $\sigma$-finite positive measure on some  space $Y$. The relative entropy of the probability measure $p$ with respect to $r$ is loosely defined by
\begin{equation}\label{eq-01app}
H(p|r):=\int_Y \log(dp/dr)\, dp\in (-\infty,\infty],\qquad p\in \PY
\end{equation}
if $p\ll r$ and $H(p|r)=\infty$ otherwise. 
More precisely, when $r$ is a probability measure,  we have $$H(p|r)=\int_Y h(dp/dr)\,dr\in[0,\infty],\qquad p,r\in\PY$$ with $h(a)=a\log a-a+1\ge 0$ for all $a\ge0,$ (take $h(0)=1).$ Hence,  the definition \eqref{eq-01app} is meaningful. 
If $r$ is unbounded, one must restrict the definition of $H(\cdot|r)$ to some subset of $\PY$ as follows. As $r$ is assumed to be $\sigma$-finite, there exist  measurable functions $W:Y\to [1,\infty)$ such that
\begin{equation}\label{eq-03}
z_W:=\int_Y e ^{-W}\, dr<\infty.
\end{equation}
Define the probability measure $r_W:= z_W ^{-1}e ^{-W}\,r$ so that $\log(dp/dr)=\log(dp/dr_W)-W-\log z_W.$ It follows that for any $p\in \PY$ satisfying $\int_Y W\, dp<\infty,$ the formula 
\begin{equation*}
H(p|r):=H(p|r_W)-\int_Y W\,dp-\log z_W\in (-\infty,\infty]
\end{equation*}
is a meaningful definition of the relative entropy which is coherent in the following sense. If $\int_Y W'\,dp<\infty$ for another measurable function $W':Y\to[0,\infty)$ such that $z_{W'}<\infty,$ then $H(p|r_W)-\int_Y W\,dp-\log z_W=H(p|r _{W'})-\int_Y W'\,dp-\log z_{W'}\in (-\infty,\infty]$.
\\
Therefore, $H(p|r)$ is well-defined for any $p\in \PY$ such that $\int_Y W\,dp<\infty$ for some measurable nonnegative function $W$ verifying \eqref{eq-03}.


\end{document}